\documentclass[a4paper,english,oneside]{article}
\usepackage[T1]{fontenc}
\usepackage[utf8]{inputenc} 
\usepackage{authblk}

\usepackage{geometry}
\usepackage{amsmath}
\usepackage{amsfonts}
\usepackage{amssymb}
\usepackage{amsthm}
\usepackage[english]{babel}
\usepackage{color}
\usepackage{lmodern}
\usepackage{pstricks-add}
\usepackage{hyperref}
\hypersetup{hidelinks}

\usepackage{mathscinet}
\usepackage{enumitem}
\usepackage[cal=boondoxo,bb=ams]{mathalfa}
\usepackage{microtype}
\usepackage{mathtools}
\usepackage{esint}

\newtheorem{theorem}{Theorem}[section]

\newtheorem{definition}{Definition}[section]
\newtheorem{lemma}{Lemma}[section]
\newtheorem{proposition}{Proposition}[section]
\newtheorem{remark}{Remark}[section]

\numberwithin{equation}{section}

\begin{document}
\title
{Blow-up results for a Nakao-type problem with a time-dependent damping term and derivative-type nonlinearities} 

\author[a,b]{Yuequn Li}
\author[b]{Alessandro Palmieri\footnote{
\href{mailto:alessandro.palmieri@uniba.it}{alessandro.palmieri@uniba.it}}}

\affil[a]{School of Mathematical Sciences, Nanjing Normal University Nanjing 210023, China}
\affil[b]{Department of Mathematics, University of Bari, Via E. Orabona 4, 70125 Bari, Italy}

\date{}

\maketitle


\begin{abstract}
In this paper, we consider a semilinear system of damped wave equations coupled through power nonlinearities of derivative-type. In particular, we consider a classical damped wave equation, i.e., with constant coefficients, and a wave equation with a time-dependent coefficient for the damping term. For this time-dependent coefficient we analyze two cases: the scale-invariant case and the scattering producing case. We prove blow-up results and derive upper bound estimates for the lifespan of local solutions. Our approach is based on an iteration argument for a couple of functionals related to the components of a local solution.
\end{abstract}


\begin{flushleft}
\end{flushleft}

\begin{flushleft}
\textbf{AMS Classification (2020)} 35B44, 35G55, 35L52, 35L71.
\end{flushleft}

\section{Introduction}\label{s1}
In this paper, we consider the damped Nakao-type problem with derivative-type nonlinearities
\begin{align}\label{eqs}
\begin{cases}
\partial_t^2u-\Delta u+\partial_tu=\vert \partial_tv\vert^p, & t>0,\, x\in\mathbb{R}^n, \\
\partial_t^2v-\Delta v+b(t)\partial_tv=\vert \partial_tu\vert^q, & t>0, \,  x\in\mathbb{R}^n,\\
(u,\partial_tu,v,\partial_tv)(0,x)=\varepsilon (u_0,u_1,v_0,v_1)(x), \ \ & x\in\mathbb{R}^n,
\end{cases}
\end{align}
where $p,q>1$ and $\varepsilon>0$ describes the size of the initial data.
For the coefficient $b=b(t)$ of the damping term in the $v$-equation we carry out our analysis in two cases:
\begin{enumerate}
\item the \emph{scale-invariant case}, that is, for $b(t)=\frac{\mu}{1+t}$ with $\mu>0$;
\item the \emph{scattering producing case}, that is, when $b=b(t)$ is nonnegative and $b\in L^1([0,\infty))$. 
\end{enumerate}
We provide a short overview of the above two cases later in this section.

Next, we briefly discuss the existing literature on semilinear wave and damped wave equations with a derivative-type nonlinearity and the extension of these results to weakly coupled systems.

For the semilinear wave equation
\begin{equation}\label{single1}
\begin{cases}
\partial_t^2u-\Delta u=\vert \partial_tu\vert^p, & t>0, \, x\in\mathbb{R}^n, \\
u(0,x)=\varepsilon u_0(x), & x\in\mathbb{R}^n, \\
\partial_t u(0,x)=\varepsilon u_1(x), & x\in\mathbb{R}^n,
\end{cases}
\end{equation} the critical exponent is the so-called Glassey exponent $p_{\mathrm{Gla}}(n)=\frac{n+1}{n-1}$ for $n\geq 2$ (formally, $p_{\mathrm{Gla}}(1)=+\infty$). For the proof of the Glassey's conjecture, we refer to \cite{Ag1991, HiTu1995, HiWaYo2012, Jo1981, Ma1983, Ra1987, Sc1986, Si1984, Tz1998, Wang2015, Zhou2001} and to the references therein. By critical exponent we mean here the threshold value for the exponent $p$ that separates the blow-up range $1<p\leq p_{\mathrm{Gla}}(n)$ from the small data solutions' global existence range $p>p_{\mathrm{Gla}}(n)$. In the recent work \cite{ChenPa2024}, the finer class of critical nonlinear terms $\{|\partial_t u|^{p_{\mathrm{Gla}}(n)} \mu(|\partial_t u|): \mu \mbox{ modulus of continuity} \}$ is investigated for the wave equation.

Let us consider the weakly coupled system of semilinear wave equations with derivative-type nonlinearities
\begin{equation}\label{eqs1}
\begin{cases}
\partial_t^2u-\Delta u=\vert \partial_tv\vert^p, & t>0, \, x\in\mathbb{R}^n,\\
\partial_t^2v-\Delta v=\vert \partial_tu\vert^q, & t>0, \, x\in\mathbb{R}^n,\\
(u,\partial_tu,v,\partial_tv)(0,x)=\varepsilon (u_0,u_1,v_0,v_1)(x), \ \ & x\in\mathbb{R}^n.
\end{cases}
\end{equation}
Combining the partial results in the literature \cite{De1999, IkSoWa2019, KuKu2006, Xu2004}, we have that the critical line in the $(p,q)$-plane is given by
\begin{equation}\label{criticalcurve1}
\Lambda(n,p,q)\doteq \max\left\{\frac{p+1}{pq-1},\frac{q+1}{pq-1}\right\}-\frac{n-1}{2}=0.
\end{equation} In this framework, by critical line we mean the line in the $(p,q)$-plane such that for $\Lambda(n,p,q)\geq 0$ local solutions blow up in finite time (under suitable sign assumptions for the Cauchy data), while for $\Lambda(n,p,q)< 0$ it holds the global existence of small data solutions. We point out that, to the best of our knowledge, the global existence of small data solutions to \eqref{eqs1} has been proved just for $n=3$.

For the Cauchy problem associated to the semilinar wave equation 
\begin{equation}\label{single b(t)}
\begin{cases}
\partial_t^2u-\Delta u+b(t)\partial_t u =\vert \partial_t u\vert^p, & t>0, \, x\in\mathbb{R}^n, \\
u(0,x)=\varepsilon u_0(x), & x\in\mathbb{R}^n, \\
\partial_t u(0,x)=\varepsilon u_1(x), & x\in\mathbb{R}^n,
\end{cases}
\end{equation} blow-up results have been established in the scattering producing case $b\in L^1([0,\infty))$ in \cite{LaiTa2019,WY2019} and in the scale-invariant case $b(t)=\frac{\mu}{1+t}$ in \cite{HaHa2020, HaHa2021, LaiTaWa2017, PaTu2021}. In the first case, it is shown the blow-up of local solutions when $1<p\leq p_{\mathrm{Gla}}(n)$, while, in the latter case, the blow-up range is given by $1<p\leq p_{\mathrm{Gla}}(n+\mu)$. The classification of the damping coefficient for the corresponding homogeneous problems was introduced by Wirth in \cite{Wirth2004,Wirth2006}. In particular, when $b\in L^1([0,\infty))$ we call the damping term $b(t)\partial_t u$ scattering producing due to the following result (proved in \cite[Theorem 3.1]{Wiphdt}): \emph{if $b=b(t)$ is a nonnegative and summable function, then there exists an isomorphism $\mathcal{W}:\mathcal{E}\to \mathcal{E}$, where $\mathcal{E}\doteq \dot{H}^1(\mathbb{R}^n)\times L^2(\mathbb{R}^n)$, such that, denoting by $w,\tilde{w}$ the solutions of the linear homogeneous Cauchy problems
\begin{align*}
& \partial_t^2 w-\Delta w+b(t) \partial_t w=0, && (w,\partial_t w)(0,\cdot)=(w_0,w_1), \\
& \partial_t^2 \tilde{w}-\Delta \tilde{w}=0, && (\tilde{w},\partial_t \tilde{w})(0,\cdot)=\mathcal{W}(w_0,w_1),
\end{align*} respectively, then
\begin{align*}
\|(w,\partial_t w)(t,\cdot)-(\tilde{w},\partial_t \tilde{w})(t,\cdot)\|_{\mathcal{E}} \to 0 \ \ \mbox{as} \ \ t\to +\infty.
\end{align*}}
Roughly speaking, we can say that in the scattering producing case, the lower order term $b(t)\partial_t u$ is somewhat negligible and does not modify the critical exponent from \eqref{single1} to \eqref{single b(t)}, while in the scale-invariant case, there is shift in the space dimension in the Glassey exponent: this shift is a consequence of the threshold nature of the scale-invariant coefficients in Wirth's classification. In the case with classical damping $b(t)=1$, to the best of our knowledge, no blow-up result for \eqref{single b(t)} has been proved in the literature.

Let us discuss now the state of the art for the following weakly coupled system of semilinear damped wave equations
\begin{equation}\label{eqs2}
\begin{cases}
\partial_t^2u-\Delta u+b_1(t)\partial_tu=\vert \partial_tv\vert^p, & t>0,\, x\in\mathbb{R}^n,\\
\partial_t^2v-\Delta v+b_2(t)\partial_tv=\vert \partial_tu\vert^q, & t>0,\, x\in\mathbb{R}^n,\\
(u,\partial_tu,v,\partial_tv)(0,x)=\varepsilon (u_0,u_1,v_0,v_1)(x),  & x\in\mathbb{R}^n.
\end{cases}
\end{equation}
When $b_1,b_2$ are both nonnegative $L^{1}([0,\infty))$ functions, the local solutions to \eqref{eqs2} blow up for $p,q>1$ such that $\Lambda(n,p,q)\geq 0$, see \cite{PaTa20202}. On the other, for $b_j(t)=\frac{\mu_j}{1+t}$ with $j=1,2$, combining the results from \cite{HaHa2022,PaTu2021} we have that blow-up results hold for \eqref{eqs2} for $p,q>1$ such that
\begin{equation}\label{newregion}
\max\left\{\frac{p+1}{pq-1}-\frac{n+\mu_1-1}{2},\ \frac{q+1}{pq-1}-\frac{n+\mu_2-1}{2}\right\}\geq 0.
\end{equation} Clearly, for $\mu_1=\mu_2$ the previous blow-up range is exactly $\Lambda(n+\mu_1,p,q)\geq 0$.

For $b_1(t)=1$ and $b_2(t)=0$, the Cauchy problem in \eqref{eqs2} is also known as Nakao-type problem, after the author of \cite{N16,N18}. In \cite{ChenReissig2021, Waka2017} the corresponding weakly coupled system with power nonlinearities has been studied, obtaining as blow-up range an intermediate region between those for the weakly coupled systems of wave equations and the weakly coupled systems of damped wave equations with the same kind of nonlinear terms. In \cite{Chen2022}, Chen proved the blow-up of local solutions to \eqref{eqs2} for $b_1(t)=1$ and $b_2(t)=0$ (under suitable sign assumptions for the data) whenever the exponents $p,q>1$ satisfy
\begin{equation}\label{blow-up range Nakao}
\Theta(n,p,q) \doteq \frac{1}{pq-1}-\frac{n-1}{2}>0,
\end{equation} by following the approach introduced in \cite[Section 13.2]{LiZh2017}. Subsequently, in \cite{PaTa2023} the blow-up for the local solutions to \eqref{eqs2} for $(b_1,b_2)=(1,0)$ has been proved also in the limit case $\Theta(n,p,q)=0$ by employing the blow-up technique along a certain characteristic line from \cite{Zhou2001}.

Finally, in \cite{LiPa2025} we have recently studied the Nakao-type problem in \eqref{eqs} with power nonlinearities $(|v|^p,|u|^q)$ in place of the derivative-type nonlinearities $(|\partial_t v|^p,|\partial_t u|^q)$ (both for the scale-invariant and the scattering producing case).
 
 In this work, by introducing suitable functionals and deriving the corresponding lower bound estimates, we apply an iteration argument to establish the blow-up region in the $(p,q)$ plane in the two considered cases for $b$. For the scattering producing case, we will prove a blow-up result for \eqref{eqs} in the same range as in the case $b(t)=0$, while in the scale invariant case the blow-up condition will be $\Theta(n+\mu,p,q)\geq 0$, finding also for \eqref{eqs} a shift of magnitude $\mu$ in the space dimension.

   The remaining sections of this paper are organized as follows: in Section 2, we present the definition of weak solution to \eqref{eqs} and state our main results.  In Section 3, we introduce the relevant functionals and derive the associated iteration frames and first lower bound estimates. Finally, in Section 4 we prove the main results by using an iteration argument.

\paragraph*{Notations}
$a \lesssim b$  means that there exists a constant $C>0$ independent of $t,x,\varepsilon$ such that $a\leq Cb$. 
$B_R\doteq\{x\in\mathbb{R}^n: |x|\leq R\}$ denotes the ball of radius $R>0$ around the origin.

\section{Main results}

Let us begin by introducing the definition of the weak solutions to \eqref{eqs} that we employ in the present paper. We underline that this class of solutions is the largest class to which our approach can be applied. For the sake of brevity and simplicity, we keep calling these solutions weak solutions, although we are assuming more regularity with respect to the space variable in comparison to the usual definition of weak solutions.

\begin{definition}\label{engdef1}
Let $\varepsilon>0$ and $u_0,v_0\in W^{1,1}_{\mathrm{loc}}(\mathbb{R}^n)$, $u_1,v_1\in L^1_{\mathrm{loc}}(\mathbb{R}^n)$. We assume that $u_0,u_1,v_0,v_1$ are compactly supported functions with 
\begin{align}\label{support condition data}
\mathrm{supp}\, (u_0,u_1,v_0,v_1)\subset B_R
\end{align} 
for some $R>0$. We say that $(u,v)$ is a weak solution to \eqref{eqs} on $[0,T)$, if
\begin{align*}
 & u \in \mathcal{C}\big([0,T);W^{1,1}_{\mathrm{loc}}(\mathbb{R}^n)\big)\cap \mathcal{C}^1\big([0,T);L^1_{\mathrm{loc}}(\mathbb{R}^n)\big) \ \mbox{such that }  \partial_t u\in L^q_{\mathrm{loc}}([0,T)\times\mathbb{R}^n) \\
 & v \in \mathcal{C}\big([0,T);W^{1,1}_{\mathrm{loc}}(\mathbb{R}^n)\big)\cap \mathcal{C}^1\big([0,T);L^1_{\mathrm{loc}}(\mathbb{R}^n)\big) \ \mbox{such that }  \partial_t v\in L^p_{\mathrm{loc}}([0,T)\times\mathbb{R}^n) 
\end{align*} 
satisfy the support condition
\begin{align}\label{support condition sol}
\mathrm{supp}\, (u,v)(t,\cdot) \subset B_{R+t} \ \mbox{ for any} \ t\in(0,T),
\end{align} 
the initial conditions $u(0,\cdot) = \varepsilon u_0$, $v(0,\cdot) = \varepsilon v_0$ in $L_{\mathrm{loc}}^1(\mathbb{R}^n)$, and the following integral identities
\begin{align}
&\int_0^t \int_{\mathbb{R}^n}(-\partial_s u(s,x)\partial_s\phi(s,x)+\partial_su(s,x)\phi(s,x)+\nabla u(s,x)\cdot\nabla\phi(s,x))\mathrm{d}x\mathrm{d}s \notag \\
&+\int_{\mathbb{R}^n}\partial_tu(t,x)\phi(t,x)\mathrm{d}x=\varepsilon\int_{\mathbb{R}^n}u_1(x)\phi(0,x)\mathrm{d}x
+\int_0^t\int_{\mathbb{R}^n}\vert \partial_sv(s,x)\vert^p\phi(s,x)\mathrm{d}x\mathrm{d}s \label{inteeq11}
\end{align} and
\begin{align}
&\int_0^t\int_{\mathbb{R}^n}\big(-\partial_sv(s,x)\partial_s\psi(s,x)+\nabla v(s,x)\cdot\nabla\psi(s,x)+b(s) \partial_sv(s,x)\psi(s,x)\big)\mathrm{d}x\mathrm{d}s \notag \\
&+\int_{\mathbb{R}^n}\partial_tv(t,x)\psi(t,x)\mathrm{d}x=\varepsilon\int_{\mathbb{R}^n}v_1(x)\psi(0,x)\mathrm{d}x
+\int_0^t\int_{\mathbb{R}^n}\vert \partial_su(s,x)\vert^q\psi(s,x)\mathrm{d}x\mathrm{d}s \label{inteeq12}
\end{align}
for any $t\in(0,T)$ and any test functions $\phi,\psi\in \mathcal{C}_0^\infty([0,T)\times\mathbb{R}^n)$.
\end{definition}

\begin{remark}\label{furtherintegrate}
Performing further integrations by parts in \eqref{inteeq11} and \eqref{inteeq12}, we have
\begin{align}
&\int_0^t\int_{\mathbb{R}^n}u(s,x)\big(\partial^2_s\phi(s,x)-\Delta\phi(s,x)-\partial_s\phi(s,x)\big)\mathrm{d}x\mathrm{d}s \notag\\
&\quad+\int_{\mathbb{R}^n}\big(\partial_tu(t,x)\phi(t,x)-u(t,x)\partial_t\phi(t,x)+u(t,x)\phi(t,x)\big)\mathrm{d}x \notag\\
&=\varepsilon\int_{\mathbb{R}^n}\big((u_0(x)+u_1(x))\phi(0,x)-u_0(x)\partial_t\phi(0,x)\big)\mathrm{d}x+\int_0^t\int_{\mathbb{R}^n}\vert \partial_sv(s,x)\vert^p\phi(s,x)\mathrm{d}x\mathrm{d}s \label{inteeq21}
\end{align}
and
\begin{align}
&\int_0^t\int_{\mathbb{R}^n}v(s,x)\left[\partial_s^2\psi(s,x)-\Delta\psi(s,x)-\partial_s\big(b(s)\psi(s,x)\big)\right]\mathrm{d}x\mathrm{d}s \notag\\
&\quad +\int_{\mathbb{R}^n}\big(\partial_tv(t,x)\psi(t,x)-v(t,x)\partial_t\psi(t,x)+b(t)v(t,x)\psi(t,x)\big)\mathrm{d}x\notag \\
&=\varepsilon\int_{\mathbb{R}^n}\big((b(0) v_0(x)+v_1(x))\psi(0,x)-v_0(x)\partial_t\psi(0,x)\big)\mathrm{d}x+\int_0^t\int_{\mathbb{R}^n}\vert \partial_su(s,x)\vert^q\psi(s,x)\mathrm{d}x\mathrm{d}s \label{inteeq22}
\end{align}
for any $t\in(0,T)$ and any test functions $\phi,\psi\in \mathcal{C}_0^\infty([0,T)\times\mathbb{R}^n)$.
\end{remark}

\begin{theorem}\label{theorem1}
Let $b(t)=\frac{\mu}{1+t}$, for some $\mu>0$ in \eqref{eqs}.  
Let $u_0,v_0\in W^{1,1}_{\mathrm{loc}}(\mathbb{R}^n), u_1,v_1\in L^{1}_{\mathrm{loc}}(\mathbb{R}^n)$ be nonnegative and compactly supported functions satisfying \eqref{support condition data} for some $R>0$ such that
 $u_1\geq u_0$ and $v_0$ is nontrivial. 
Let us assume that the exponents $p,q>1$ fulfill
\begin{equation}\label{method2}
\Theta(n+\mu,p,q)=\frac{1}{pq-1}-\frac{n+\mu-1}{2}\geq 0.
\end{equation} 
 Then, there exists $\varepsilon_0=\varepsilon_0(n,\mu,p,q,R,
 v_0,v_1)>0$ such that for any $\varepsilon\in(0,\varepsilon_0]$  the local weak solution $(u,v)$ to \eqref{eqs} on $[0,T(\varepsilon))$, in the sense of Definition \ref{engdef1}, blows up in finite time and
\begin{equation}
T(\varepsilon)\leq 
\begin{cases}
C\varepsilon^{-\tfrac{1}{\Theta(n+\mu,p,q)}} & \text{if} \ \, \Theta(n+\mu,p,q)>0,\\
\exp\left(C\varepsilon^{-(pq-1)}\right) & \text{if}\ \, \Theta(n+\mu,p,q)=0,
\end{cases}
\end{equation} where $C>0$ is independent of $\varepsilon$.

\end{theorem}

\begin{theorem}\label{theorem2}
Let $b\in L^1([0,\infty))$ be a nonnegative function.
Let $u_0,v_0\in W^{1,1}_{\mathrm{loc}}(\mathbb{R}^n), u_1,v_1\in L^{1}_{\mathrm{loc}}(\mathbb{R}^n)$ be nonnegative and compactly supported functions satisfying \eqref{support condition data} for some $R>0$ such that
$u_1\geq u_0$ and $v_1$ is nontrivial.

Let us consider $p,q>1$ such that 
\begin{equation}\label{blowupcondition2}
\Theta(n,p,q)=\frac{1}{pq-1}-\frac{n-1}{2}\geq 0.
\end{equation}

Then, there exists $\varepsilon_0=\varepsilon_0(n,b,p,q,R,v_1)>0$ such that for any $\varepsilon\in(0,\varepsilon_0]$  the local weak solution $(u,v)$ to \eqref{eqs} on $[0,T(\varepsilon))$, in the sense of Definition \ref{engdef1}, blows up in finite time and
\begin{equation}
T(\varepsilon)\leq 
\begin{cases}
C\varepsilon^{-\tfrac{1}{\Theta(n,p,q)}} & \text{if} \ \, \Theta(n,p,q)>0,\\
\exp\left(C\varepsilon^{-(pq-1)}\right) & \text{if}\ \, \Theta(n,p,q)=0,
\end{cases}
\end{equation} where $C>0$ is independent of $\varepsilon$.
\end{theorem}

\begin{remark} The blow-up condition for \eqref{eqs} in the scale-invariant case in Theorem \ref{theorem1} is obtained by the condition $\Theta(n+\mu,p,q)\geq 0$ which is a shifted condition with respect to the blow-up condition for the Nakao problem, i.e. \eqref{eqs} when $b(t)=0$, obtained in \cite{Chen2022, PaTa2023}. On the other hand, for the scattering producing case the blow-up condition $\Theta(n,p,q)\geq 0$ is not altered by the presence of the damping term $b(t)\partial_t v$. Therefore, we found that for the Cauchy problem \eqref{eqs} the presence of a damping term in the $v$-equation (both in the scale-invariant case and in the scattering producing case) affects the blow-up range in an analogous way as for the corresponding single semilinear damped wave equation and weakly coupled systems of damped wave equations, that we discussed in the introduction.
\end{remark}

\section{Deduction of the iteration frame}
\subsection{Scale-invariant damping}\label{seccase1}

In this subsection, we consider the case $b(t)=\frac{\mu}{1+t}$.
Let us consider the adjoint equation to the linearized homogeneous version of the $v$-equation in \eqref{eqs}, namely,
\begin{equation}\label{conjugateeq}
\frac{\partial^2\psi}{\partial s^2 }(s,x)-\Delta\psi(s,x)-\frac{\partial}{\partial s}\Big(\frac{\mu}{1+s} \psi(s,x)\Big)=0.
\end{equation}
We consider a solution to \eqref{conjugateeq} with separate variables. 

Let $\varphi$ be the function introduced in \cite{Zh2006}, that is,
\begin{align} \label{phi YZ06}
\varphi(x)\doteq
\begin{cases}
  \mathrm{e}^{x}+\mathrm{e}^{-x}  & \mbox{if} \ n=1, \\
  \vphantom{\bigg(} \displaystyle{\int_{\mathbb{S}^{n-1}}\mathrm{e}^{x\cdot \omega}\mathrm{d}\sigma_\omega} & \mbox{if} \ n\geqslant 2.
 \end{cases}
\end{align} The function $\varphi$ satisfies the following properties: $\varphi \in \mathcal{C}^\infty(\mathbb{R}^n)$ is a positive function such that $\Delta \varphi =\varphi$ and $\varphi(x)\sim c_n |x|^{-\frac{n-1}{2}}\mathrm{e}^{|x|}$ as $|x|\to +\infty$. 

As time-dependent function, we consider the function $\rho$
\begin{align} \label{def rho(t)}
\rho(t)=(1+t)^\frac{\mu+1}{2}\mathrm{K}_{\ell}(1+t),
\end{align} where $\ell\doteq \frac{\vert\mu-1\vert}{2}$ and $\mathrm{K}_\ell$ denotes the modified Bessel function of the second kind of order $\ell$.
Since the function $\rho$ is a solution of the ODE
\begin{align}\label{rho equation}
\frac{\mathrm{d}^2\rho}{\mathrm{d} s^2 }(s)-\rho(s)-\frac{\mathrm{d}}{\mathrm{d} s}\Big(\frac{\mu}{1+s} \rho(s)\Big)=0,
\end{align} cf. \cite[Lemma 3.2]{LiPa2025}, it follows that the function 
\begin{equation}\label{psi2}
\Psi_2(s,x)\doteq \rho(s)\varphi(x)
\end{equation}
 is a solution to \eqref{conjugateeq}.
 We recall now some properties of the function $\mathrm{K}_\nu$, that we are going to use in the next computations. The integral representation
\begin{equation}\label{knu}
\mathrm{K}_\nu(t)=\int_0^\infty \mathrm{e}^{-t\cosh y}\cosh(\nu y)\mathrm{d}y, \ \nu\in\mathbb{R}
\end{equation} holds for any $t>0$ and any $\nu\in\mathbb{R}$ \cite[Equation
(10.32.9)]{Tricomi1953}, moreover, the asymptotic behavior for large times and the recursive relation for the derivative are given by
\begin{align}
& \mathrm{K}_\nu(t)=\sqrt{\frac{\pi}{2t}}\mathrm{e}^{-t}\bigl(1+O(t^{-1})\bigr)\qquad as \ t\rightarrow\infty,\label{prop1}\\
& \frac{\mathrm{d} \mathrm{K}_\nu}{\mathrm{d}t}(t)=-\mathrm{K}_{\nu+1}(t)+\frac{\nu}{t}\mathrm{K}_\nu(t)\label{prop2},
\end{align} see \cite[Equations (10.40.2) and (10.29.2)]{Tricomi1953}.

By using the asymptotic behavior of $\varphi$, we obtain the estimate in the next lemma.

\begin{lemma}\label{lemma3.3}
Let $r\geq 1$ and $R>0$, then, there exists a $C=C(n,r,R)>0$ such that
\begin{equation}
\int_{B_{R+t}}(\varphi(x))^{r}\mathrm{d}x\leq C\mathrm{e}^{rt}(R+t)^{(n-1)(1-\frac{r}{2})} 
\end{equation} for any $t\geq 0$.
\end{lemma}

Let us consider $\varepsilon>0$ and $u_0,u_1,v_0,v_1$ satisfying the assumptions of the statement of Theorem \ref{theorem1}. From now on in this subsection, $(u,v)$ is a given local solution  to \eqref{eqs} on $[0,T)$. We consider the following functionals
\begin{align*}
& && U_0(t)\doteq \int_{\mathbb{R}^n}u(t,x)\Psi_1(t,x)\mathrm{d}x, && V_0(t)\doteq \int_{\mathbb{R}^n}v(t,x)\Psi_2(t,x)\mathrm{d}x, && \\
& && U_1(t)\doteq \int_{\mathbb{R}^n}\partial_tu(t,x)\Psi_1(t,x)\mathrm{d}x, && V_1(t)\doteq \int_{\mathbb{R}^n}\partial_tv(t,x)\Psi_2(t,x)\mathrm{d}x, && 
\end{align*}
where $\Psi_1(t,x)\doteq \mathrm{e}^{-t}\varphi(x)$ is the solution of the free wave equation $\partial_t^2\Psi_1-\Delta\Psi_1=0$ introduced for the first time in \cite{Zh2006} to study the critical semilinear wave equation, and $\Psi_2$ is defined by \eqref{psi2}.
We work fist with the functionals associated to $u$. Hence, we begin by proving the nonnegativity of $U_0$.
\begin{proposition}\label{nonnegativity}
Let $u_0\in W^{1,1}_{\mathrm{loc}}(\mathbb{R}^n), u_1\in L^1_{\mathrm{loc}}(\mathbb{R}^n)$ be compactly supported functions such that $U_0(0),U'_0(0)\geq 0$. Then, $U_0(t)\geq 0$ for any $t\in[0,T)$.
\end{proposition}

\begin{proof}
If we perform a further step of integration by parts, by \eqref{inteeq11} we get
\begin{align}
&\int_0^t \int_{\mathbb{R}^n}-\left[(\partial_s u(s,x)+u(s,x))\partial_s\phi(s,x)+ u(s,x)\Delta\phi(s,x)\right]\mathrm{d}x\mathrm{d}s +\int_{\mathbb{R}^n}(\partial_tu(t,x)+u(t,x))\phi(t,x)\mathrm{d}x \notag \\
& \qquad =\varepsilon\int_{\mathbb{R}^n}(u_1(x)+u_0(x))\phi(0,x)\mathrm{d}x
+\int_0^t\int_{\mathbb{R}^n}\vert \partial_sv(s,x)\vert^p\phi(s,x)\mathrm{d}x\mathrm{d}s.\label{inteeq13}
\end{align}
We choose $\phi=\Psi_1$ as test function in \eqref{inteeq13}. We stress that this can be done even if $\Psi_1$ is not compactly supported, thanks to the support condition \eqref{support condition sol}. Since $\partial_t\Psi_1=-\Psi_1$ and $\Delta\Psi_1=\Psi_1$, it results
\begin{align*}
&\int_0^t\int_{\mathbb{R}^n}\partial_su(s,x)\Psi_1(s,x)\mathrm{d}x\mathrm{d}s+\int_{\mathbb{R}^n}\big(\partial_tu(t,x)+u(t,x)\big)
\Psi_1(t,x)\mathrm{d}x\\
&=\varepsilon\int_{\mathbb{R}^n}\big(u_0(x)+u_1(x)\big)\varphi(x)\mathrm{d}x+\int_0^t\int_{\mathbb{R}^n}
\vert\partial_sv(s,x)\vert^p\Psi_1(s,x)\mathrm{d}x\mathrm{d}s,
\end{align*} that is,
\begin{align}\label{6}
\int_0^t U_1(s) \mathrm{d}s+U_1(t)+U_0(t)= \varepsilon\int_{\mathbb{R}^n}\big(u_0(x)+u_1(x)\big)\varphi(x)\mathrm{d}x+\int_0^t\int_{\mathbb{R}^n}
\vert\partial_sv(s,x)\vert^p\Psi_1(s,x)\mathrm{d}x\mathrm{d}s.
\end{align}
Since $U_1(t)=U_0^\prime(t)+U_0(t)$, \eqref{6} can be rewritten as
\begin{align}\label{8}
&\int_0^tU_0(s)\mathrm{d}s+U_0^\prime(t)+3U_0(t)
=\varepsilon\int_{\mathbb{R}^n}\big(2u_0(x)+u_1(x)\big)\varphi(x)\mathrm{d}x +\int_0^t\int_{\mathbb{R}^n}
\vert\partial_sv(s,x)\vert^p\Psi_1(s,x)\mathrm{d}x\mathrm{d}s.
\end{align}
Differentiating \eqref{8} with respect to $t$, we have
\begin{equation}\label{9}
U_0^{\prime\prime}(t)+3U_0^\prime(t)+U_0(t)=\int_{\mathbb{R}^n}
\vert\partial_tv(t,x)\vert^p\Psi_1(t,x)\mathrm{d}x.
\end{equation}
We factorize the second-order differential operator $\frac{\mathrm{d}^2}{\mathrm{d}t^2}+3\frac{\mathrm{d}}{\mathrm{d}t}+1$ as follows
\begin{align*}
\mathrm{e}^{-\alpha_1t}\frac{\mathrm{d}}{\mathrm{d}t}\Big(\mathrm{e}^{(\alpha_1-\alpha_2)t}\frac{\mathrm{d}}{\mathrm{d}t}\big(\mathrm{e}^{\alpha_2t}U_0(t)\big)\Big)
&=U_0^{\prime\prime}(t)+(\alpha_1+\alpha_2)U_0^\prime(t)+\alpha_1\alpha_2U_0(t)\\
&=U_0^{\prime\prime}(t)+3U_0^\prime(t)+U_0(t),
\end{align*}
where $\alpha_1,\alpha_2$ solve the quadratic equation $\alpha^2-3\alpha+1=0$. 
From \eqref{9}, it follows that
\begin{equation}\label{91}
\mathrm{e}^{-\alpha_1t}\frac{\mathrm{d}}{\mathrm{d}t}\Big(\mathrm{e}^{(\alpha_1-\alpha_2)t}\frac{d}{dt}\big(\mathrm{e}^{\alpha_2t}U_0(t)\big)\Big)
=\int_{\mathbb{R}^n}
\vert\partial_tv(t,x)\vert^p\Psi_1(t,x)\mathrm{d}x.
\end{equation}
From \eqref{91}, elementary algebraic steps and an integration over $[0,t]$ lead to 
\begin{align*}
&\frac{\mathrm{d}}{\mathrm{d}t}\big(\mathrm{e}^{\alpha_2t}U_0(t)\big)=\mathrm{e}^{(\alpha_2-\alpha_1)t}\big(U_0^\prime(0)+\alpha_2U_0(0)\big)
+\mathrm{e}^{(\alpha_2-\alpha_1)t}\int_0^t\mathrm{e}^{\alpha_1s}\int_{\mathbb{R}^n}\vert\partial_s v(s,x)\vert^p\Psi_1(s,x)\mathrm{d}x\mathrm{d}s.
\end{align*} By integrating over $[0,t]$ the last equation, we obtain
\begin{align*}
 \mathrm{e}^{\alpha_2t}U_0(t)-U_0(0) & =\frac{\mathrm{e}^{(\alpha_2-\alpha_1)t}-1}{\alpha_2-\alpha_1}\big(U_0^\prime(0)+\alpha_2U_0(0)\big) \\ & \qquad +\int_0^t\mathrm{e}^{(\alpha_2-\alpha_1)s}\int_0^s\mathrm{e}^{\alpha_1\tau}\int_{\mathbb{R}^n}\vert\partial_\tau v(\tau,x)\vert^p\Psi_1(\tau,x)\mathrm{d}x\mathrm{d}\tau \mathrm{d}x.
\end{align*}
Hence,
\begin{align}
U_0(t) & =\frac{\alpha_1\mathrm{e}^{-\alpha_2t}-\alpha_2\mathrm{e}^{-\alpha_1t}}{\alpha_1-\alpha_2}U_0(0)+
\frac{\mathrm{e}^{-\alpha_2t}-\mathrm{e}^{-\alpha_1t}}{\alpha_1-\alpha_2}U_0^\prime(0) \notag\\
& \qquad +\mathrm{e}^{-\alpha_2t}\int_0^t\mathrm{e}^{(\alpha_2-\alpha_1)s}\int_0^s\mathrm{e}^{\alpha_1\tau}\int_{\mathbb{R}^n}\vert\partial_\tau v(\tau,x)\vert^p\Psi_1(\tau,x)\mathrm{d}x\mathrm{d}\tau \mathrm{d}s.\label{representation U0}
\end{align}
Since the $t$-dependent functions multiplying $U_0(0)$ and $U_0'(0)$ in the previous equation are nonnegative, from \eqref{representation U0} it follows the nonnegativity of $U_0$. We underline that, in this last step, we used the assumptions on the data $U_0(0),U'_0(0)\geq 0$ and the fact that the nonlinear term $|\partial_t v|^p$ and the test function $\Psi_1$ are nonnegative.
\end{proof}

\begin{proposition}\label{proposition2}
Let $u_0\in W^{1,1}_{\mathrm{loc}}(\mathbb{R}^n), u_1\in L^1_{\mathrm{loc}}(\mathbb{R}^n)$ be compactly supported functions such that $U_0(0),U'_0(0), U_1(0)\geq 0$. Then, for any $t\in[0,T)$
\begin{equation}\label{u1t}
U_1(t)\geq\int_0^t\mathrm{e}^{2(s-t)}\int_{\mathbb{R}^n}\vert\partial_sv(s,x)\vert^p\Psi_1(s,x)\mathrm{d}x\mathrm{d}s.
\end{equation}
\end{proposition}
\begin{proof}
By using  $U_1(t)=U_0^\prime(t)+U_0(t)$, we can rewrite \eqref{9} as
\begin{equation*}
U_1^\prime(t)+2U_1(t)-U_0(t)=\int_{\mathbb{R}^n}\vert\partial_tv(t,x)\vert^p\Psi_1(t,x)\mathrm{d}x.
\end{equation*}
Thus,
\begin{equation}\label{re9}
\mathrm{e}^{-2t}\frac{\mathrm{d}}{\mathrm{d}t}\big(\mathrm{e}^{2t}U_1(t)\big)=U_0(t)+\int_{\mathbb{R}^n}\vert\partial_tv(t,x)\vert^p\Psi_1(t,x)\mathrm{d}x.
\end{equation}
From \eqref{re9}, we arrive at
\begin{equation*}
U_1(t)= \mathrm{e}^{-2t}U_1(0)+\mathrm{e}^{-2t}\int_0^t\mathrm{e}^{2s}U_0(s)\mathrm{d}s+\mathrm{e}^{-2t}\int_0^t\mathrm{e}^{2s}\int_{\mathbb{R}^n}
\vert\partial_sv(s,x)\vert^p\Psi_1(s,x)\mathrm{d}x\mathrm{d}s,
\end{equation*}
therefore, \eqref{u1t} follows immediately from the nonnegativity of $U_1(0)$ and Proposition \ref{nonnegativity}.
\end{proof}

Next, we consider the functionals associated with $v$. In the next results, we follow the approach from \cite{HaHa2021} to deal with the scale-invariant damping term. By employing the asymptotic behavior of $\mathrm{K}_\nu(t)$ as $t\rightarrow\infty$, we have
\begin{align}
\rho(t)=\sqrt{\frac{\pi}{2\mathrm{e}^2}}(1+t)^{\frac{\mu}{2}}\mathrm{e}^{-t}\Big(1+O\big((1+t)^{-1}\big)\Big) \qquad  \text{as}\ t\rightarrow\infty.\label{12}
\end{align}
Consequently, there exist $T_1>0$ and $C_1>0$ such that for any $t\geq T_1$
\begin{align}
\frac{1}{C_1}\mathrm{e}^{-2t}\leq \rho^2(t)(1+t)^{-\mu}\leq C_1\mathrm{e}^{-2t}.\label{15}
\end{align}

First, we prove two preliminary lemmas.
\begin{lemma}\label{lem1} Let $\rho$ be the function definde in \eqref{def rho(t)}. Then, for any $t\geq 0:$ $$\frac{\mu}{1+t}-\frac{\rho^\prime(t)}{\rho(t)}>0.$$
\end{lemma}
\begin{proof}
By using the recursive relation in \eqref{prop2}, we find that
\begin{equation*}
\begin{aligned}
\rho^\prime(t)&=\tfrac{\mu+1}{2}(1+t)^{\frac{\mu-1}{2}}\mathrm{K}_\ell(1+t)+(1+t)^{\frac{\mu+1}{2}}
\Big(-\mathrm{K}_{\ell+1}(1+t)+\tfrac{\vert\mu-1\vert}{2(1+t)}\mathrm{K}_\ell(1+t)\Big)\\
&=(1+t)^{\frac{\mu+1}{2}}\Big(\tfrac{\mu+1+\vert\mu-1\vert}{2(1+t)}\mathrm{K}_\ell(1+t)-\mathrm{K}_{\ell+1}(1+t)\Big),
\end{aligned}
\end{equation*}
where we used that $\ell=\frac{\vert\mu-1\vert}{2}$. Then,
\begin{align}\label{rho'/rho}
\frac{\rho^\prime(t)}{\rho(t)}= \frac{\mu+1+\vert\mu-1\vert}{2(1+t)}-\frac{\mathrm{K}_{\ell+1}(1+t)}{\mathrm{K}_{\ell}(1+t)}
\end{align}
and, consequently,
\begin{equation}\label{relar4}
\frac{\mu}{1+t}-\frac{\rho^\prime(t)}{\rho(t)}=\frac{\mu-1-\vert\mu-1\vert}{2(1+t)}+\frac{\mathrm{K}_{\ell+1}(1+t)}{\mathrm{K}_\ell(1+t)}.
\end{equation}
By \eqref{knu} we see that the function $\mathrm{K}_\nu(t)$ is increasing with respect to $\nu$ (fixed $t>0$), therefore, $\frac{\mathrm{K}_{\ell+1}(1+t)}{\mathrm{K}_\ell(1+t)}\geq 1$ for any $t\geqslant 0$. On the other hand, for any $t\geq 0$, when $ \mu\geq 1$ we have
\begin{align*}
\frac{\mu-1-\vert\mu-1\vert}{2(1+t)}=0,
\end{align*}  while when $\mu\in (0,1)$ it holds
\begin{align*} 
\frac{\mu-1-\vert\mu-1\vert}{2(1+t)}=\frac{\mu-1}{1+t}\geq \mu-1>-1.
\end{align*}
Summarizing, by \eqref{relar4} we conclude that
\begin{equation*} 
\frac{\mu}{1+t}-\frac{\rho^\prime(t)}{\rho(t)}> -1+\frac{\mathrm{K}_{\ell+1}(1+t)}{\mathrm{K}_\ell(1+t)}\geq0,
\end{equation*} which is the desired inequality.
\end{proof}

\begin{lemma}\label{le33}
Let $\rho$ be the function defined in \eqref{def rho(t)}.
Then, there exists $T_2>0$ such that $$1\leq\frac{\mu}{1+t}-\frac{2\rho^\prime(t)}{\rho(t)}\leq 3$$ for any $t\geq T_2$.
\end{lemma}
\begin{proof}
It is a straightforward consequence of the fact that
\begin{equation*}
\frac{\mu}{1+t}-\frac{2\rho^\prime(t)}{\rho(t)}=-\frac{1+\vert\mu-1\vert}{1+t}+
\frac{2\mathrm{K}_{\ell+1}(1+t)}{\mathrm{K}_\ell(1+t)}\longrightarrow  2\qquad \text{as}\ t\rightarrow\infty,
\end{equation*}
where we used \eqref{rho'/rho} and \eqref{prop1}.
\end{proof}

\begin{proposition}\label{propo3}
Let $v_0\in W^{1,1}_{\mathrm{loc}}(\mathbb{R}^n), v_1\in L^1_{\mathrm{loc}}(\mathbb{R}^n)$ be nonnegative and compactly supported functions such that $v_0$ is nontrivial. Then, for any $t\in[0,T)$
\begin{equation}\label{23}
V_0(t)\geq K_0\, \varepsilon,
\end{equation}
 where $K_0=K_0(n,v_0,v_1,\mu)>0$ is independent of $\varepsilon$.
\end{proposition}
\begin{proof}
Being  $\Psi_2$ a solution to \eqref{conjugateeq},  choosing $\psi=\Psi_2$ into \eqref{inteeq22} we obtain
\begin{align*}
&\int_{\mathbb{R}^n}\left(\partial_t v(t,x)\Psi_2(t,x)-v(t,x)\partial_t\Psi_2(t,x)+\frac{\mu}{1+t}v(t,x)\Psi_2(t,x)\right)\mathrm{d}x\\
&=\varepsilon\rho(0)\int_{\mathbb{R}^n}\left(v_1(x)+\left(\mu-\frac{\rho^\prime(0)}{\rho(0)}\right)v_0(x)\right)\varphi (x)\mathrm{d}x+
\int_0^t\int_{\mathbb{R}^n}\vert \partial_su(s,x)\vert^q\Psi_2(s,x)\mathrm{d}x\mathrm{d}s.
\end{align*}
The previous equation can be rewritten as
\begin{align}
V_0^\prime(t)+\left(\frac{\mu}{1+t}-\frac{2\rho^\prime(t)}{\rho(t)}\right)V_0(t)  & =\varepsilon I_\mu[v_0,v_1]+
\int_0^t\int_{\mathbb{R}^n}\vert \partial_su(s,x)\vert^q\Psi_2(s,x)\mathrm{d}x\mathrm{d}s, \label{25}
\end{align}
where 
\begin{align*}
I_\mu[v_0,v_1]\doteq \rho(0)\int_{\mathbb{R}^n}\left(v_1(x)+\left(\mu-\frac{\rho^\prime(0)}{\rho(0)}\right)v_0(x)\right)\varphi (x)\mathrm{d}x. 
\end{align*}
Using Lemma \ref{lem1} and the assumptions on the functions $v_0,v_1$, we have $I_\mu[v_0,v_1]>0$. 
From \eqref{25}, we arrive at
\begin{equation*}
\frac{\rho^2(t)}{(1+t)^\mu}\frac{\mathrm{d}}{\mathrm{d}t}\left(\frac{(1+t)^\mu}{\rho^2(t)}V_0(t)\right)=\varepsilon I_\mu[v_0,v_1]+
\int_0^t\int_{\mathbb{R}^n}\vert \partial_su(s,x)\vert^q\Psi_2(s,x)\mathrm{d}x\mathrm{d}s.
\end{equation*}
After some elementary computations and an integrating over $[0,t]$, we find
\begin{align}
V_0(t)&=\frac{V_0(0)}{\rho^2(0)}\frac{\rho^2(t)}{(1+t)^\mu}+\varepsilon I_\mu[v_0,v_1]\frac{\rho^2(t)}{(1+t)^\mu}\int_0^t\frac{(1+s)^\mu}{\rho^2(s)}\mathrm{d}s \notag \\
& \qquad +\frac{\rho^2(t)}{(1+t)^\mu}\int_0^t\frac{(1+s)^\mu}{\rho^2(s)}\int_0^s\int_{\mathbb{R}^n}\vert\partial_\tau
u(\tau,x)\vert^q\Psi_2(\tau,x)\mathrm{d}x\mathrm{d}\tau \mathrm{d}s. \label{26}
\end{align} 
By the continuity and positivity of the function $\frac{\rho^2(t)}{(1+t)^\mu}$ on $[0,2T_1]$ (recall that $\mathrm{K}_\ell(1+t)$ is decreasing with respect to $t$ and tends to $0$ as $t\to +\infty$), from \eqref{26} it follows that
\begin{equation}\label{27}
V_0(t)\geq\frac{V_0(0)}{\rho^2(0)}\frac{\rho^2(t)}{(1+t)^\mu}\gtrsim\varepsilon \int_{\mathbb{R}^n}v_0(x)\varphi(x)\mathrm{d}x
\end{equation} for any $t\in[0,2T_1]$, where the nonnegativity and the nontriviality of $v_0$ guarantee that the integral on the right-hand side of \eqref{27} is positive.

On the other hand, for any $t\in[2T_1,T)$,  by \eqref{15} it follows that
\begin{equation}\label{28}
\begin{aligned}
V_0(t)&\geq\varepsilon I_\mu[v_0,v_1]\frac{\rho^2(t)}{(1+t)^\mu}\int_0^t\frac{(1+s)^\mu}{\rho^2(s)}\mathrm{d}s\\
&\geq\varepsilon I_\mu[v_0,v_1]\frac{\rho^2(t)}{(1+t)^\mu}\int_{\frac{t}{2}}^t\frac{(1+s)^\mu}{\rho^2(s)}\mathrm{d}s\\
&\gtrsim\varepsilon I_\mu[v_0,v_1]\, \mathrm{e}^{-2t}\int_{\frac{t}{2}}^t\mathrm{e}^{2s}\mathrm{d}s\gtrsim \varepsilon I_\mu[v_0,v_1].
\end{aligned}
\end{equation}
Combining \eqref{27} and \eqref{28}, we conclude the validity of \eqref{23}.
\end{proof}

\begin{proposition}\label{propo4}
Let $v_0\in W^{1,1}_{\mathrm{loc}}(\mathbb{R}^n), v_1\in L^1_{\mathrm{loc}}(\mathbb{R}^n)$ be nonnegative and compactly supported functions such that $v_0$ is nontrivial. Then, there exists $T_3>0$ and $K_1=K_1(n,\mu,v_0,v_1)>0$ such that for any $ t\in [T_3,T)$
\begin{equation}\label{29}
V_1(t)\geq K_1\varepsilon.
\end{equation}
\end{proposition}

\begin{proof}
By the definitions of $V_0(t),V_1(t)$, we have
\begin{equation}\label{30}
V_0^\prime(t)=V_1(t)+\frac{\rho^\prime(t)}{\rho(t)}V_0(t).
\end{equation}
Hence, \eqref{25} can be rewritten as
\begin{equation}\label{31}
V_1(t)+\left(\frac{\mu}{1+t}-\frac{\rho^\prime(t)}{\rho(t)}\right)V_0(t)=\varepsilon I_\mu[v_0,v_1]+
\int_0^t\int_{\mathbb{R}^n}\vert \partial_su(s,x)\vert^q\Psi_2(s,x)\mathrm{d}x\mathrm{d}s.
\end{equation}
Differentiating \eqref{31} with respect to $t$ we get
\begin{align*}
& V_1^\prime(t)+\left(\frac{\mu}{1+t}-\frac{\rho^\prime(t)}{\rho(t)}\right)V_0^\prime(t)+
\left(\left(\frac{\rho^\prime(t)}{\rho(t)}\right)^2-\frac{\mu}{(1+t)^2}-\frac{\rho^{\prime\prime}(t)}{\rho(t)}
\right)V_0(t) =\int_{\mathbb{R}^n}\vert\partial_tu(t,x)\vert^q\Psi_2(t,x)\mathrm{d}x.
\end{align*}
Since $\rho$ satisfies \eqref{rho equation}, we conclude that
\begin{align}
\int_{\mathbb{R}^n}\vert\partial_tu(t,x)\vert^q\Psi_2(t,x)\mathrm{d}x & = 
V_1^\prime(t)+\left(\frac{\mu}{1+t}-\frac{\rho^\prime(t)}{\rho(t)}\right)V_0^\prime(t)+
\left(\left(\frac{\rho^\prime(t)}{\rho(t)}\right)^2-1-\frac{\mu}{1+t}\frac{\rho'(t)}{\rho(t)}
\right)V_0(t) \notag
\\ &= V_1^\prime(t)+\left(\frac{\mu}{1+t}-\frac{\rho^\prime(t)}{\rho(t)}\right)\left(V_0^\prime(t)- \frac{\rho^\prime(t)}{\rho(t)}V_0(t)\right)-V_0(t) \notag \\
& = V_1^\prime(t)+\left(\frac{\mu}{1+t}-\frac{\rho^\prime(t)}{\rho(t)}\right)V_1(t)-V_0(t), \label{32}
\end{align} 
where in the third equality we used \eqref{30}.

Next, we rewrite the right-hand side of \eqref{32}. Fix $\omega\in(\frac{1}{2},1)$, then,
\begin{align}
&V_1^\prime(t)+\left(\frac{\mu}{1+t}-\frac{\rho^\prime(t)}{\rho(t)}\right)V_1(t)-V_0(t) \notag \\ &\quad =V_1^\prime(t)+\omega
\left(\frac{\mu}{1+t}-\frac{2\rho^\prime(t)}{\rho(t)}\right)V_1(t) -g_1(t)\left(V_1(t)+\left(\frac{\mu}{1+t}-\frac{\rho^\prime(t)}{\rho(t)}\right)V_0(t)\right)-g_2(t)V_0(t), \label{33}
\end{align}
where
\begin{align*}
g_1(t)& \doteq (\omega-1)\frac{\mu}{1+t}+(1-2\omega)\frac{\rho^\prime(t)}{\rho(t)},\\
g_2(t)& \doteq 1-\left((\omega-1)
\frac{\mu}{1+t}+(1-2\omega)\frac{\rho^\prime(t)}{\rho(t)}\right)\left(\frac{\mu}{1+t} -\frac{\rho^\prime(t)}{\rho(t)}\right).
\end{align*}
By \eqref{rho'/rho} and \eqref{prop1}, we find that
\begin{equation*}
\frac{\rho^\prime(t)}{\rho(t)}=\frac{\mu+1+\vert \mu-1\vert}{2(1+t)}-\frac{1+O\big((1+t)^{-1}\big)}{1+O\big((1+t)^{-1}\big)}\longrightarrow-1
\end{equation*} 
as $t\rightarrow +\infty$. Therefore,
\begin{equation*}
\lim_{t\to +\infty}g_1(t)=2\omega-1,\ \lim_{t\to +\infty}g_2(t)=2-2\omega.
\end{equation*}
Consequently, there exists $T_4>0$ such that
\begin{equation}\label{39}
g_1(t)>\omega-\tfrac{1}{2},\quad \ g_2(t)>1-\omega,
\end{equation} for any $t\geq T_4$.
By \eqref{32} and \eqref{33}, we have
\begin{align}
&V_1^\prime(t)+\omega
\left(\frac{\mu}{1+t}-\frac{2\rho^\prime(t)}{\rho(t)}\right)V_1(t) \notag \\
& \quad =g_2(t)V_0(t) +g_1(t)\left(V_1(t)+\left(\frac{\mu}{1+t}-\frac{\rho^\prime(t)}{\rho(t)}\right)V_0(t)\right)
+\int_{\mathbb{R}^n}\vert\partial_tu(t,x)\vert^q\Psi_2(t,x)\mathrm{d}x. \label{V'1+ omega ... V1}
\end{align}
Then, combining by \eqref{23},  \eqref{31}, \eqref{39} and \eqref{V'1+ omega ... V1},  for any $t\in[T_4,T)$ it results
\begin{align}
&V_1^\prime(t)+\omega
\left(\frac{\mu}{1+t}-\frac{2\rho^\prime(t)}{\rho(t)}\right)V_1(t) \notag \\
& \quad \geq (1-\omega)K_0\, \varepsilon+  (\omega-\tfrac{1}{2})\left(\varepsilon I_\mu[v_0,v_1]+
\int_0^t\int_{\mathbb{R}^n}\vert \partial_su(s,x)\vert^q\Psi_2(s,x)\mathrm{d}x\mathrm{d}s\right) +\int_{\mathbb{R}^n}\vert\partial_tu(t,x)\vert^q\Psi_2(t,x)\mathrm{d}x \notag \\
& \quad \geq  \left((1-\omega)K_0+(\omega-\tfrac{1}{2})I_\mu[v_0,v_1]\right)\varepsilon,  \label{36}
\end{align}
 where in the last step we used the nonneativity of the nonlinear term and of $\Psi_2$.
 
 Hence, from \eqref{36}, for $t\in[T_4,T)$ we get
\begin{equation*}
\frac{\rho(t)^{2\omega}}{(1+t)^{\mu\omega}}\frac{\mathrm{d}}{\mathrm{d}t}\left(\frac{(1+t)^{\mu\omega}}
{\rho(t)^{2\omega}}V_1(t)\right)\gtrsim\varepsilon.
\end{equation*}
Integrating the last inequality over $[T_5,t]$, where $T_5\doteq \max\{T_1,T_4\}$, we arrive at
\begin{equation}\label{38}
V_1(t)\gtrsim\frac{(1+T_5)^{\mu\omega}}{\rho(T_5)^{2\omega}}\frac{\rho(t)^{2\omega}}{(1+t)^{\mu\omega}}
V_1(T_5)+\varepsilon\frac{\rho(t)^{2\omega}}{(1+t)^{\mu\omega}}\int_{T_5}^t\frac{(1+s)^{\mu\omega}}
{\rho(s)^{2\omega}}\mathrm{d}s.
\end{equation}
In order to derive a lower bound estimate for $V_1$ from \eqref{38}, we first show that $V_1(t)$ is nonnegative for any $t\in[0,T)$. To achieve this goal, we introduce an auxiliary functional $G(t)\doteq \int_{\mathbb{R}^n}v(t,x)\varphi(x)\mathrm{d}x$ for any $t\in [0,T)$. Clearly, $V_1(t)=\rho(t)G^\prime(t)$, so if we prove that $G^\prime(t)\geq0$ for any $t\in[0,T)$, then, it follows the nonnegativity of $V_1$ on $[0,T)$.

 Taking $\varphi(x)$ as the test function in \eqref{inteeq22}, we get
\begin{align*}
&G^\prime(t)+\frac{\mu}{1+t}G(t)-\int_0^tG(s)\mathrm{d}s+\int_0^t\frac{\mu}{(1+s)^2}G(s)\mathrm{d}s\\
& \quad =\varepsilon\int_{\mathbb{R}^n}\big(v_1(x)+\mu v_0(x)\big)\mathrm{d}x+\int_0^t\int_{\mathbb{R}^n}\vert\partial_su(s,x)\vert^q\varphi(x)\mathrm{d}x\mathrm{d}s.
\end{align*}
Differentiating the last equation with respect to $t$, we obtain
\begin{equation}\label{382}
G^{\prime\prime}(t)+\frac{\mu}{1+t}G^\prime(t)=G(t)+\int_{\mathbb{R}^n}\vert\partial_tu(t,x)
\vert^q\varphi(x)\mathrm{d}x.
\end{equation}
Since $V_0(t)=\rho(t)G(t)$, by Proposition \ref{propo3} it follows that $G(t)>0$ for any $t\in[0,T)$.
Therefore, from \eqref{382} it follows that
\begin{equation*}
(1+t)^{-\mu}\frac{\mathrm{d}}{\mathrm{d}t}\Big((1+t)^\mu G^\prime(t)\Big)\geq\int_{\mathbb{R}^n}\vert\partial_tu(t,x)\vert^q\varphi(x)\mathrm{d}x 
\end{equation*} for any $t\in[0,T)$.
By straightforward computations, we have
\begin{align*}
G^\prime(t)&\geq (1+t)^{-\mu}G^\prime(0)+(1+t)^{-\mu}\int_0^t(1+s)^\mu\int_{\mathbb{R}^n}\vert\partial_su(s,x)\vert^q
\varphi(x)\mathrm{d}x\mathrm{d}s\\
&\geq(1+t)^{-\mu}G^\prime(0)\\
& =\varepsilon(1+t)^{-\mu}\int_{\mathbb{R}^n}v_1(x)\varphi(x)\mathrm{d}x\geq0
\end{align*} for any $t\in[0, T)$. Thus, $V_1(t)\geq0$ for any $t\in[0,T)$.

Finally, by \eqref{38} and \eqref{15},  for any $t\in[2T_5,T)$ we find
\begin{align*}
V_1(t)&\gtrsim\varepsilon\frac{\rho(t)^{2\omega}}{(1+t)^{\mu\omega}}\int_{T_5}^t\frac{(1+s)^{\mu\omega}}
{\rho(s)^{2\omega}}\mathrm{d}s\\
&\geq\frac{\varepsilon}{C_1^{2\omega}} \mathrm{e}^{-2\omega t}\int_{\frac{t}{2}}^t\mathrm{e}^{2\omega s}\mathrm{d}s=\frac{\varepsilon}{2\omega C_1^{2\omega}}(1-\mathrm{e}^{-\omega t}) \geq \frac{1-\mathrm{e}^{-2\omega T_5}}{2\omega C_1^{2\omega}}\, \varepsilon.
\end{align*}
Setting $T_3\doteq 2T_5$ and taking $0<K_1\lesssim\frac{1-\mathrm{e}^{-2\omega T_5}}{2\omega C_1^{2\omega}}$, 
we completed the proof of \eqref{29}.
\end{proof}

Next, we determine the iterative frame for $(U_1,V_1)$.
\begin{proposition}\label{case1pro}
Let $u_0,v_0\in W^{1,1}_{\mathrm{loc}}(\mathbb{R}^n), u_1,v_1\in L^{1}_{\mathrm{loc}}(\mathbb{R}^n)$ be nonnegative and compactly supported functions. Then, there exist $T_6>0$ and $C, K>0$ such that  for any $t\in[T_6,t)$
\begin{align}
&U_1(t)\geq C\int_{T_6}^t\mathrm{e}^{2(s-t)}(1+s)^{-\frac{n-1}{2}(p-1)-\frac{\mu}{2}p}(V_1(s))^p\mathrm{d}s,\label{if1}\\
&V_1(t)\geq K\int_{T_6}^t(1+s)^{-\frac{n-1}{2}(q-1)+\frac{\mu}{2}}(U_1(s))^q\mathrm{d}s.\label{if2}
\end{align} 
\end{proposition}
\begin{proof}
We are going to derive the iteration frame by using Proposition  \ref{proposition2} and \eqref{V'1+ omega ... V1}-\eqref{36}. Applying  H\"older's inequality, we have
\begin{align}\label{holder1}
&V_1(t) \leq \left(\int_{\mathbb{R}^n}\vert\partial_tv(t,x)\vert^p\Psi_1(t,x)\mathrm{d}x\right)^{\frac{1}{p}}
\left(\int_{B_{R+t}}\frac{(\Psi_2(t,x))^{p^\prime}}{(\Psi_1(t,x))^{\frac{p^\prime}{p}}}\, \mathrm{d}x\right)^{\frac{1}{p^\prime}},
\end{align}
where $\frac{1}{p}+\frac{1}{p^\prime}=1$. Let us determine an upper bound  estimate for the integral over $R_{R+t}$ on the right-hand side of \eqref{holder1}. By \eqref{12}, we get
\begin{align*}
&\frac{(\Psi_2(t,x))^{p^\prime}}{(\Psi_1(t,x))^{\frac{p^\prime}{p}}}=\rho(t)^{p^\prime}\mathrm{e}^{\frac{p^\prime}{p}t}\varphi(x)
=\left(\tfrac{\pi}{2\mathrm{e}^2}\right)^{\frac{p^\prime}{2}}\mathrm{e}^{-t}(1+t)^{\frac{\mu}{2}p^\prime}\varphi(x)\left(1+O\left((1+t)^{-1}\right)\right)\ \quad \text{as}\ t\to+\infty.
\end{align*}
In particular, from this asymptotic behavior, it follows that there exists $T_7>0$ such that for any $t\in[T_7,T)$ the following estimate holds: $$(\Psi_2(t,x))^{p^\prime}(\Psi_1(t,x))^{-\frac{p^\prime}{p}}\lesssim(1+t)^{\frac{\mu}{2}p^\prime}\mathrm{e}^{-t}\varphi(x).$$
Thus, by Lemma \ref{lemma3.3}, for any $t\in[T_7,T)$ we have
\begin{align*}
\int_{B_{R+t}}\frac{(\Psi_2(t,x))^{p^\prime}}{(\Psi_1(t,x))^{\frac{p^\prime}{p}}}\,\mathrm{d}x & \lesssim(1+t)^{\frac{\mu}{2}p^\prime}\mathrm{e}^{-t}
\int_{B_{R+t}}\varphi(x)\mathrm{d}x\\ & \lesssim(1+t)^{\frac{\mu}{2}p^\prime+\frac{n-1}{2}}.
\end{align*}
Consequently, employing Proposition  \ref{proposition2} and  \eqref{holder1}, for any $t\in[T_7,T)$ we find
\begin{align}
U_1(t)&\geq\int_{T_7}^t\mathrm{e}^{2(s-t)}\int_{\mathbb{R}^n}\vert\partial_sv(s,x)\vert^p\Psi_1(s,x)\mathrm{d}x\mathrm{d}s \notag\\
&\gtrsim\int_{T_7}^t\mathrm{e}^{2(s-t)}(1+s)^{-\frac{\mu}{2}p-\frac{n-1}{2}(p-1)}(V_1(s))^p\mathrm{d}s. \label{it1}
\end{align}

In order to derive the estimate in \eqref{if2}, we introduce the auxiliary functional
\begin{equation*}
F(t):=V_1(t)-C_2\int_{T_8}^t\int_{\mathbb{R}^n}\vert\partial_su(s,x)\vert^q\Psi_2(s,x)\mathrm{d}x\mathrm{d}s-C_3\varepsilon
\end{equation*}
with $C_2,C_3>0$ positive constants, whose ranges will be fixed afterwards, and $T_8\doteq \max\{T_2,T_3,T_4\}$. 
Since
\begin{align*}
& F^\prime(t)+\omega\left(\frac{\mu}{1+t}-\frac{2\rho^\prime(t)}{\rho(t)}\right)F(t)\\ & \quad =V_1^\prime(t) +\omega\left(\frac{\mu}{1+t}-\frac{2\rho^\prime(t)}{\rho(t)}\right)V_1(t) -C_2\int_{\mathbb{R}^n}\vert\partial_tu(t,x)\vert^q\Psi_2(t,x)\mathrm{d}x\\ & \qquad -\omega\left(\frac{\mu}{1+t}
-\frac{2\rho^\prime(t)}{\rho(t)}\right)\left(C_ 2\int_{T_8}^t\int_{\mathbb{R}^n} \vert\partial_su(s,x)\vert^q\Psi_2(s,x)\mathrm{d}x\mathrm{d}s +C_3\varepsilon \right),
\end{align*}
by \eqref{36} and Lemma \ref{le33}, for $t\in[T_8,T)$ we have
\begin{align*}
& F^\prime(t)+\omega\left(\frac{\mu}{1+t}-\frac{2\rho^\prime(t)}{\rho(t)}\right)F(t)\\ 
& \quad \geq  (1-\omega)K_0\varepsilon + \left(\omega-\tfrac{1}{2} \right)\left(\varepsilon I_\mu[v_0,v_1]+
\int_0^t\int_{\mathbb{R}^n}\vert \partial_su(s,x)\vert^q\Psi_2(s,x)\mathrm{d}x\mathrm{d}s\right) \\
&\qquad  +(1-C_2)\int_{\mathbb{R}^n}\vert\partial_tu(t,x)\vert^q\Psi_2(t,x)\mathrm{d}x \\
&\qquad   -\omega\left(\frac{\mu}{1+t}
-\frac{2\rho^\prime(t)}{\rho(t)}\right)\left(C_ 2\int_{T_8}^t\int_{\mathbb{R}^n} \vert\partial_su(s,x)\vert^q\Psi_2(s,x)\mathrm{d}x\mathrm{d}s +C_3\varepsilon \right)\\
&\quad \geq\Big((\omega-\tfrac{1}{2})I_\mu[v_0,v_1]+(1-\omega)K_0-3\omega C_3\Big)\varepsilon+
(1-C_2)\int_{\mathbb{R}^n}\vert\partial_tu(t,x)\vert^q\Psi_2(t,x)\mathrm{d}x\\
&\qquad+\big((\omega-\tfrac{1}{2})-3C_2\omega\big)\int_{T_8}^t\int_{\mathbb{R}^n}
\vert\partial_su(s,x)\vert^q\Psi_2(s,x)\mathrm{d}x\mathrm{d}s.
\end{align*}
Let us fix $C_2,C_3$ such that
\begin{align*}
0<C_2<\min\left\{\tfrac{1}{3\omega}\left(\omega-\tfrac{1}{2}\right),1\right\}, \ 0<C_3<\tfrac{1}{3\omega}\left(\left(\omega-\tfrac{1}{2}\right)I_\mu[v_0,v_1]+(1-\omega)K_0\right).
\end{align*}
Thanks to these choices for $C_2,C_3$, by the nonnegativity of the nonlinear term and $\Psi_2$, for any $t\in[T_8,T)$ it follows that 
\begin{align*}
&\frac{\rho(t)^{2\omega}}{(1+t)^{\mu\omega}}\frac{\mathrm{d}}{\mathrm{d}t}\left(\frac{(1+t)^{\mu\omega}}
{\rho(t)^{2\omega}}F(t)\right) = F^\prime(t)+\omega\left(\frac{\mu}{1+t}-\frac{2\rho^\prime(t)}{\rho(t)}\right)F(t)\geq 0.
\end{align*}
Thus, for any $t\in[T_8,T)$ it holds
\begin{align}\label{40}
F(t)\geq\frac{(1+T_8)^{\mu\omega}}{(1+t)^{\mu\omega}}\frac{\rho(t)^{2\omega}}{\rho(T_8)^{2\omega}}F(T_8).
\end{align}
If $C_3<K_1$, then, by Proposition \ref{propo4}, we have $F(T_8)=V_1(T_8)-C_3\varepsilon\geq (K_1-C_3)\varepsilon>0$. Then, by \eqref{40} it follows that $F$ is nonnegative in $[T_8,T)$. Hence, for any $t\in[T_8,T)$
\begin{align} \label{40.1}
 V_1(t)\geq C_2\int_{T_8}^t\int_{\mathbb{R}^n}\vert\partial_su(s,x)\vert^q\Psi_2(s,x)\mathrm{d}x\mathrm{d}s.
\end{align}
Next, we apply H\"older's inequality and \eqref{support condition sol}
\begin{align} \label{41}
U_1(t)=\int_{\mathbb{R}^n}\partial_tu(t,x)\Psi_1(t,x)\mathrm{d}x 
\leq\left(\int_{\mathbb{R}^n}\vert\partial_tu(t,x)\vert^q\Psi_2(t,x)\mathrm{d}x\right)^{\frac{1}{q}}
\left(\int_{B_{R+t}}\frac{(\Psi_1(t,x))^{q^\prime}}{(\Psi_2(t,x))^{\frac{q^\prime}{q}}}\mathrm{d}x\right)^{\frac{1}{q^\prime}}.
\end{align}
By \eqref{12}, there exists $T_9>0$ such that
\begin{align*}
\frac{(\Psi_1(t,x))^{q^\prime}}{(\Psi_2(t,x))^{\frac{q^\prime}{q}}} &=\rho(t)^{-\frac{q^\prime}{q}}\mathrm{e}^{-q^\prime t}\varphi(x) 
\lesssim (1+t)^{-\frac{\mu q^\prime}{2q}}\mathrm{e}^{-t}\varphi(x).
\end{align*}
 for any $t\in[T_9,T)$. Hence, by Lemma \ref{lemma3.3}, we have
\begin{align}
\int_{\mathbb{R}^n}\frac{(\Psi_1(t,x))^{q^\prime}}{(\Psi_2(t,x))^{\frac{q^\prime}{q}}}\mathrm{d}x\lesssim (1+t)^{-\frac{\mu q^\prime}{2q}+\frac{n-1}{2}} \label{42}
\end{align} for any $t\in [T_9,T)$.
Combining \eqref{40.1}, \eqref{41} and \eqref{42}, we obtain 
\begin{equation}\label{it2}
V_1(t)\gtrsim\int_{\max\{T_8,T_9\}}^t(1+s)^{\frac{\mu}{2}-\frac{n-1}{2}(q-1)}(U_1(s))^q\mathrm{d}s 
\end{equation} for any $ t\in [\max\{T_8,T_9\},T)$.
Choosing $T_6\doteq \max\{T_7,T_8,T_9\}$, by \eqref{it1} and \eqref{it2}, we completed the proof.
\end{proof}

\subsection{Scattering producing damping}

In this subsection, we consider the case $b\geq 0$, $b\in L^1([0,\infty))$. 
Let us consider $\varepsilon>0$ and $u_0,u_1,v_0,v_1$ satisfying the assumptions of the statement of Theorem \ref{theorem2}. Let $(u,v)$  be a local solution  to \eqref{eqs} on $[0,T)$. We consider the following functionals
\begin{align*}
& && U_0(t)\doteq \int_{\mathbb{R}^n}u(t,x)\Psi_1(t,x)\mathrm{d}x, && V_0(t)\doteq \int_{\mathbb{R}^n}v(t,x)\Psi_1(t,x)\mathrm{d}x, && \\
& && U_1(t)\doteq \int_{\mathbb{R}^n}\partial_tu(t,x)\Psi_1(t,x)\mathrm{d}x, && V_1(t)\doteq \int_{\mathbb{R}^n}\partial_tv(t,x)\Psi_1(t,x)\mathrm{d}x, && 
\end{align*}
where $\Psi_1(t,x)=\mathrm{e}^{-t}\varphi(x)$ as in Section \ref{seccase1}.
In the following we adapt the main ideas from \cite{LaiTa2019}  to our model.

Since the coefficient $b=b(t)$ does not effect the linear operator on the first equation in \eqref{eqs}, Proposition \ref{nonnegativity} and Proposition \ref{proposition2} remain valid even in this case. 

Now, we focus on establishing lower bound estimates for  $V_0(t)$ and $V_1(t)$. As in \cite{LaiTa2019}, we introduce the function
\begin{equation}\label{defm}
m(t)\doteq \mathrm{e}^{-\int_t^\infty b(s)\mathrm{d}s} \qquad \mbox{for any } t\geq 0.
\end{equation}
Clearly, $m$ is an increasing function satisfying
\begin{align}
& 0< \mathrm{e}^{-\|b\|_{L^1([0,\infty))}}=m(0)\leq m(t)\leq1, \label{prop m boundedness} \\
&  m^\prime(t)=b(t) m(t), \label{prop m derivative}
\end{align} for any $t\geq 0$.

\begin{proposition}\label{9-1}
Let $v_0\in W^{1,1}_{\mathrm{loc}}(\mathbb{R}^n)$, $v_1\in L^{1}_{\mathrm{loc}}(\mathbb{R}^n)$ be nonnegative and compactly supported functions such that $v_1$ is nontrivial. Then, $V_0(t)\geq 0$ for any $t\in [0,T)$.
\end{proposition}

\begin{proof}
Since $\partial_t\Psi_1=-\Psi_1$, $\Delta\Psi_1=\Psi_1$, employing $\Psi_1$ as  test function into \eqref{inteeq12}, we get
\begin{align*}
& \int_0^t \int_{\mathbb{R}^n}\left(\partial_s v(s,x)-v(s,x)+b(s)\partial_s v(s,x)\right)\Psi_1(s,x) \mathrm{d}x \mathrm{d}s + \int_{\mathbb{R}^n}\partial_t v(t,x) \Psi_1(t,x) \mathrm{d}x \\ & \quad = \varepsilon \int_{\mathbb{R}^n} v_1(x) \varphi(x) \mathrm{d}x+\int_0^t \int_{\mathbb{R}^n}|\partial_s u(s,x)|^q\Psi_1(s,x)\mathrm{d}x\mathrm{d}s.
\end{align*}
Differentiating the previous equation with respect to $t$, we find
\begin{align}
&\frac{\mathrm{d}}{\mathrm{d}t}\int_{\mathbb{R}^n}\partial_tv(t,x)\Psi_1(t,x)\mathrm{d}x+b(t)\int_{\mathbb{R}^n}\partial_tv(t,x)\Psi_1(t,x)\mathrm{d}x \notag\\
&+\int_{\mathbb{R}^n}\big(\partial_tv(t,x)-v(t,x)\big)\Psi_1(t,x)\mathrm{d}x=\int_{\mathbb{R}^n}\vert\partial_tu(t,x)\vert^q\Psi_1(t,x)\mathrm{d}x. \label{5-1}
\end{align}
Multiplying both sides of  \eqref{5-1} by $m(t)$ and using \eqref{prop m derivative}, we obtain
\begin{align}
&\frac{\mathrm{d}}{\mathrm{d}t}\Big(m(t)\int_{\mathbb{R}^n}\partial_tv(t,x)\Psi_1(t,x)\mathrm{d}x\Big)+
m(t)\int_{\mathbb{R}^n}\big(\partial_tv(t,x)-v(t,x)\big)\Psi_1(t,x)\mathrm{d}x\notag\\
&\quad\quad=m(t)\int_{\mathbb{R}^n}\vert\partial_tu(t,x)\vert^q\Psi_1(t,x)\mathrm{d}x. \label{5-2-0}
\end{align}
Integrating  \eqref{5-2-0} over $[0,t]$, we arrive at
\begin{align} &m(t)\int_{\mathbb{R}^n}\partial_tv(t,x)\Psi_1(t,x)\mathrm{d}x-m(0)\varepsilon\int_{\mathbb{R}^n}v_1(x)\varphi(x)\mathrm{d}x\notag\\
&\qquad+\int_0^tm(s)\int_{\mathbb{R}^n}\big(\partial_sv(s,x)-v(s,x)\big)\Psi_1(s,x)\mathrm{d}x\mathrm{d}s \notag \\
&\ =\int_0^tm(s)\int_{\mathbb{R}^n}\vert\partial_su(s,x)\vert^q\Psi_1(s,x)\mathrm{d}x\mathrm{d}s. \label{5-2}
\end{align}
Since $V_1(t)=V_0^\prime(t)+V_0(t)$, \eqref{5-2} can be rewritten as
\begin{equation}\label{6-2}
\begin{aligned}
&m(t)\big(V_0^\prime(t)+V_0(t)\big)+\int_0^tm(s)V_0^\prime(s)\mathrm{d}s\\
&\ \ =m(0)\varepsilon
\int_{\mathbb{R}^n}v_1(x)\varphi(x)\mathrm{d}x+\int_0^tm(s)\int_{\mathbb{R}^n}\vert\partial_su(s,x)\vert^q\Psi_1(s,x)\mathrm{d}x\mathrm{d}s
\end{aligned}
\end{equation}
Integrating by parts and using \eqref{prop m derivative}, we have
\begin{align}
\int_0^tm(s)V_0^\prime(s)\mathrm{d}s=m(t)V_0(t)-m(0)\varepsilon\int_{\mathbb{R}^n}v_0(x)\varphi(x)\mathrm{d}x-\int_0^t b(s) m(s)V_0(s)\mathrm{d}s. \label{6-2,5}
\end{align}
Then, combining \eqref{6-2} and \eqref{6-2,5},  we obtain
\begin{align}
V_0^\prime(t)+2V_0(t) &=\frac{m(0)}{m(t)}\varepsilon\int_{\mathbb{R}^n}\big(v_0(x)+v_1(x)\big)\varphi(x)\mathrm{d}x +\frac{1}{m(t)}\int_0^t b(s) m(s) V_0(s)\mathrm{d}s \notag \\
&\quad + \frac{1}{m(t)}\int_0^tm(s)\int_{\mathbb{R}^n}\vert\partial_su(s,x)\vert^q\Psi_1(s,x)\mathrm{d}x\mathrm{d}s. \label{6-3}
\end{align}
Due to the nonnegativity of the nonlinear term and of the test function $\Psi_1$, we conclude that
\begin{align*}
&\frac{\mathrm{d}}{\mathrm{d}t}\left(\mathrm{e}^{2t}V_0(t)\right)\geq m(0) J[v_0,v_1]\, \varepsilon \, \frac{\mathrm{e}^{2t}}{m(t)}+\frac{\mathrm{e}^{2t}}{m(t)}\int_0^t b(s)m(s)V_0(s)\mathrm{d}s,
\end{align*} where 
\begin{align} \label{def J[v0,v1]}
J[v_0,v_1]\doteq \int_{\mathbb{R}^n} \left(v_0(x)+v_1(x)\right)\varphi(x)\mathrm{d}x.
\end{align}
Therefore, integrating over $[0,t]$ the last estimate and using  \eqref{prop m boundedness}, we have
\begin{align*}
\mathrm{e}^{2t}V_0(t) &\geq V_0(0)+ m(0) J[v_0,v_1]\, \varepsilon 
\int_0^t\frac{\mathrm{e}^{2s}}{m(s)}\mathrm{d}s +\int_0^t\frac{\mathrm{e}^{2s}}{m(s)}\int_0^s b(\tau) m(\tau)V_0(\tau)\mathrm{d}\tau \mathrm{d}s \\
&\geq V_0(0)+m(0) J[v_0,v_1]\, \varepsilon \int_0^t\mathrm{e}^{2s}\mathrm{d}s +\int_0^t\frac{\mathrm{e}^{2s}}{m(s)}\int_0^s b(\tau) m(\tau)V_0(\tau)\mathrm{d}\tau \mathrm{d}s.
\end{align*}
Hence,
\begin{align}
V_0(t) & \geq \mathrm{e}^{-2t}V_0(0)+\frac{m(0)}{2} J[v_0,v_1]\, \varepsilon (1-\mathrm{e}^{-2t}) +\mathrm{e}^{-2t}\int_0^t\frac{\mathrm{e}^{2s}}{m(s)}\int_0^sm(\tau)b(\tau)V_0(\tau)\mathrm{d}\tau \mathrm{d}s. \label{8-1}
\end{align}
Next, we prove that from \eqref{8-1} it follows that $V_0(t)\geq 0$ for any $t\in[0,T)$. To this purpose, we use a standard comparison argument. We have to consider separately two cases, depending on whether or not $v_0$ is trivial. If $v_0$ is nontrivial, since $v_0$ is assumed nonnegative, we have $V_0(0)>0$. By the continuity of $V_0$ it follows that $V_0(t)>0$ for $t$ in a right neighborhhood of $0$. By contradiction let us assume that $V_0$ is negative at some time. Being $V_0$ continuous, this means that $V_0$ has a zero, let us call $t_1>0$ the smallest zero of $V_0$. Then, $V_0(t)>0$ for any $t\in [0,t_1)$ and $V_0(t_1)=0$. However, this contradicts \eqref{8-1}, since 
\begin{align*}
0=V_0(t_1) & \geq \mathrm{e}^{-2t_1}V_0(0)+\frac{m(0)}{2} J[v_0,v_1]\, \varepsilon (1-\mathrm{e}^{-2t_1}) +\mathrm{e}^{-2t_1}\int_0^{t_1}\frac{\mathrm{e}^{2s}}{m(s)}\int_0^sm(\tau)b(\tau)V_0(\tau)\mathrm{d}\tau \mathrm{d}s \\ & \geq \mathrm{e}^{-2t_1}V_0(0) >0.
\end{align*} If $v_0\equiv 0$, since $v_1$ is nonnegative and nontrivial by assumption, then $V_0(0)=0$ and $V_0'(0)>0$. Thanks to the continuity of $V_0'$, there exists $t_2>0$ such that $V_0'(t)>0$ for any $t\in [0,t_2)$. Therefore, $V_0$ is strictly increasing and, consequently, positive in $(0,t_2)$. Repeating a similar argument as in the first case, we find again a contradiction if we assume that there exists $t_3>t_2$ such that $V_0(t)>0$ for any $t\in (0,t_3)$ and $V_0(t_3)=0$. Indeed, by \eqref{8-1} it follows that
\begin{align*}
0=V_0(t_3) & \geq \frac{m(0)}{2} J[0,v_1]\, \varepsilon (1-\mathrm{e}^{-2t_3}) +\mathrm{e}^{-2t_3}\int_0^{t_3}\frac{\mathrm{e}^{2s}}{m(s)}\int_0^sm(\tau)b(\tau)V_0(\tau)\mathrm{d}\tau \mathrm{d}s \\ & \geq \frac{m(0)}{2}\,  \varepsilon \, \left(1-\mathrm{e}^{-2t_3}\right) \int_{\mathbb{R}^n} v_1(x) \varphi(x) \mathrm{d}x>0.
\end{align*} So, we proved that $V_0$ is nonnegative in both cases. 
\end{proof}

\begin{proposition}\label{V1escase2}
Let $v_0\in W^{1,1}_{\mathrm{loc}}(\mathbb{R}^n)$, $v_1\in L^{1}_{\mathrm{loc}}(\mathbb{R}^n)$ be nonnegative and compactly supported functions such that $v_1$ is nontrivial. Then, for any $t\in [0,T)$
\begin{equation}\label{9-2}
V_1(t)\geq K_2\, \varepsilon+K_3\int_0^t\int_{\mathbb{R}^n}\vert\partial_su(s,x)\vert^q\Psi_1(s,x)\mathrm{d}x\mathrm{d}s,
\end{equation}
where $K_2=K_2(n,b,v_1),K_3=K_3(b)>0$ are constants independent of $\varepsilon$.
\end{proposition}
\begin{proof}
First, we rewrite \eqref{5-1} as
\begin{align}
&\frac{\mathrm{d}}{\mathrm{d}t}\int_{\mathbb{R}^n}\left(\partial_tv(t,x)+v(t,x)\right)\Psi_1(t,x)\mathrm{d}x +b(t)\int_{\mathbb{R}^n}
\partial_tv(t,x)\Psi_1(t,x)\mathrm{d}x =\int_{\mathbb{R}^n}\vert\partial_tu(t,x)\vert^q\Psi_1(t,x)\mathrm{d}x. \label{10-0}
\end{align}
Multiplying both sides of \eqref{10-0} by $m(t)$ and using \eqref{prop m derivative} and  Proposition \ref{9-1}, we find
\begin{align*}
\frac{\mathrm{d}}{\mathrm{d}t}\left(m(t)\int_{\mathbb{R}^n}\left(\partial_tv(t,x)+v(t,x)\right)\Psi_1(t,x)\mathrm{d}x\right) & =m(t)b(t)V_0(t)+m(t)
\int_{\mathbb{R}^n}\vert\partial_tu(t,x)\vert^q\Psi_1(t,x)\mathrm{d}x\\
& \geq  m(t)
\int_{\mathbb{R}^n}\vert\partial_tu(t,x)\vert^q\Psi_1(t,x)\mathrm{d}x.
\end{align*}
Integrating the last estimate over $[0,t]$, we get
\begin{align}\label{11-1}
&m(t)\int_{\mathbb{R}^n}\big(\partial_tv(t,x)+v(t,x)\big)\Psi_1(t,x)\mathrm{d}x\geq m(0) J[v_0,v_1] \,\varepsilon +\int_0^tm(s)\int_{\mathbb{R}^n}\vert
\partial_su(s,x)\vert^q\Psi_1(s,x)\mathrm{d}x\mathrm{d}s.
\end{align}
By adding \eqref{5-2-0} and \eqref{11-1},  we obtain
\begin{align}
& \frac{\mathrm{d}}{\mathrm{d}t}\big(m(t)V_1(t)\big)+2m(t)V_1(t) \notag \\ & \qquad \geq m(0)J[v_0,v_1]\, \varepsilon +m(t)\int_{\mathbb{R}^n}\vert\partial_tu(t,x)\vert^q\Psi_1(t,x)\mathrm{d}x+\int_0^tm(s)\int_{\mathbb{R}^n}\vert
\partial_su(s,x)\vert^q\Psi_1(s,x)\mathrm{d}x\mathrm{d}s. \label{12-1}
\end{align}
Let us introduce the auxiliary functional
\begin{equation*}
H(t)\doteq m(t)V_1(t)-C_4\int_0^tm(s)\int_{\mathbb{R}^n}\vert\partial_su(s,x)\vert^q\Psi_1(s,x)\mathrm{d}x\mathrm{d}s-C_5\varepsilon,
\end{equation*}
where $C_4,C_5$ are constants such that
\begin{align*}
0<C_4\leq\frac{1}{2}, \qquad  0<C_5\leq\frac{m(0)}{2}\int_{\mathbb{R}^n}v_1(x)\varphi(x)\mathrm{d}x.
\end{align*} By \eqref{12-1}, we have
\begin{equation}
\begin{aligned}
H^\prime(t)+2H(t) & \geq \left(m(0) J[v_0,v_1]-2C_5\right)\varepsilon  +m(t)(1-C_4)\int_{\mathbb{R}^n}\vert\partial_tu(t,x)\vert^q\Psi_1(t,x)\mathrm{d}x\\
&\ \quad  +(1-2C_4)\int_0^tm(s)\int_{\mathbb{R}^n}
\vert\partial_su(s,x)\vert^q\Psi_1(s,x)\mathrm{d}x\mathrm{d}s \geq0.
\end{aligned}
\end{equation}
Therefore, since $\mathrm{e}^{-2t} \frac{\mathrm{d}}{\mathrm{d}t}(\mathrm{e}^{2t} H(t))\geq 0$ for any $t\in [0,T)$, it results
\begin{align*}
H(t)\geq \mathrm{e}^{-2t}H(0)= \left(m(0) V_1(0)-C_5 \varepsilon \right) \mathrm{e}^{-2t}\geq0.
\end{align*}
Then, by the definition of $H(t)$ and \eqref{prop m boundedness}, we have
\begin{equation}
\begin{aligned}
V_1(t)&\geq \frac{C_5\varepsilon}{m(t)}+\frac{C_4}{m(t)}\int_0^tm(s)\int_{\mathbb{R}^n}\vert\partial_su(s,x)\vert^q\Psi_1(s,x)
\mathrm{d}x\mathrm{d}s\\
&\geq C_5\varepsilon+C_4m(0)\int_0^t\int_{\mathbb{R}^n}\vert\partial_su(s,x)\vert^q\Psi_1(s,x)
\mathrm{d}x\mathrm{d}s.
\end{aligned}
\end{equation}
Setting $K_2\doteq C_5$ and $K_3\doteq C_4m(0)$, we completed the proof of  \eqref{9-2}.
\end{proof}
Finally, we determine the iteration frame in the scattering producing case.

\begin{proposition}\label{iterationcase2}
Let $u_0,v_0\in W^{1,1}_{\mathrm{loc}}(\mathbb{R}^n)$, $u_1,v_1 \in L^{1}_{\mathrm{loc}}(\mathbb{R}^n)$ be nonnegative and compactly supported functions such that $v_1$ is nontrivial. Then, there exist $\tilde{C}, \tilde{K}>0$ such that for any $t \in [0,T)$
\begin{align}
U_1(t)&\geq\tilde{C}\int_0^t\mathrm{e}^{2(s-t)}(1+s)^{-\frac{n-1}{2}}(V_1(s))^p\mathrm{d}s, \label{IF U1 spc}\\
V_1(t)&\geq\tilde{K}\int_0^t(1+s)^{-\frac{n-1}{2}}(U_1(s))^q\mathrm{d}s.  \label{IF V1 spc}
\end{align}
\end{proposition}

\begin{proof}
By Lemma  \ref{lemma3.3} for $r=1$, H\"older's inequality and \eqref{support condition sol}, it results
\begin{align}
(V_1(t))^p&\leq \int_{\mathbb{R}^n}\vert\partial_tv(t,x)\vert^p\Psi_1(t,x)\mathrm{d}x\left(\int_{B_{R+t}}\Psi_1(t,x)\mathrm{d}x\right)^{p-1}\notag\\
&\lesssim (1+t)^{\frac{n-1}{2}(p-1)}\int_{\mathbb{R}^n}\vert\partial_tv(t,x)\vert^p\Psi_1(t,x)\mathrm{d}x,\label{holder21}\\
(U_1(t))^q&\leq\int_{\mathbb{R}^n}\vert\partial_tu(t,x)\vert^q\Psi_1(t,x)\mathrm{d}x\left(\int_{B_{R+t}} \Psi_1(t,x)\mathrm{d}x\right)^{q-1}\notag\\
&\lesssim (1+t)^{\frac{n-1}{2}(q-1)}\int_{\mathbb{R}^n}\vert\partial_tu(t,x)\vert^q\Psi_1(t,x)\mathrm{d}x.\label{holder22}
\end{align}
Plugging \eqref{holder21} into \eqref{u1t} and \eqref{holder22} into \eqref{9-2}, respectively, we conclude the validity of \eqref{IF U1 spc} and \eqref{IF V1 spc}.
\end{proof}

\section{Proof of the main results: Iteration argument}

\subsection{Proof of Theorem \ref{theorem1}}\label{proofofthem1}

Relying on the lower bound for $V_1(t)$ and the iteration frame established in Propositions \ref{propo4} and  \ref{case1pro},  we proceed to prove Theorem \ref{theorem1} by applying an iteration argument. The choice of the parameters for the slicing procedure is inspired by \cite[Subsection 4.3]{ChenPa2020} and \cite[Subsection 4.1]{PaTa2023}.

\subsubsection*{Subcritical case}
This section is devoted to the proof of the blow-up result in the subcritical case $\Theta(n+\mu,p,q)>0$. We set $T_0\doteq \max\{T_3,T_6\}$. Let us
introduce the sequence $\{\ell_j\}_{j\in\mathbb{N}}$ such that
\begin{align}\label{def ell j}
\ell_0\doteq \max\left\{\frac{1}{2T_0},1\right\}, \qquad \ell_j\doteq 1+(pq)^{-j} \quad \mbox{for any } j\geq 1.
\end{align}
As parameters characterizing the slicing procedure we take $\{L_j\}_{j\in\mathbb{N}}$, where 
\begin{align}\label{def Lj}
L_j& \doteq \prod_{k=0}^{j} \ell_k
\end{align} for any $j\in\mathbb{N}$. Let us underline that the infinite product $\prod_{k=0}^\infty \ell_k$ is convergent, so, if we denote by $L\doteq \prod_{k=0}^\infty \ell_k$, then, $L_j\uparrow L$ as $j\to +\infty$ (being $L_j>1$ for any $j\geq 1$).

Next, we prove that $V_1$ satisfies the following sequence of lower bound estimates: for any $j\in\mathbb{N}$
\begin{equation}\label{1-10-1}
V_1(t)\geq C_j \varepsilon^{(pq)^j} (1+t)^{-\alpha_j}(t-L_jT_0)^{\beta_j} \quad \text{for any}\ \ t\in [L_jT_0,T),
\end{equation}
where $\{C_j\}_{j\in\mathbb{N}}, \{\alpha_j\}_{j\in\mathbb{N}},\{\beta_j\}_{j\in\mathbb{N}}$ are sequences of nonnegative real numbers that will be determined inductively. 

Clearly, \eqref{29} implies that \eqref{1-10-1} holds for $j=0$, provided that  $C_0\doteq K_1,\alpha_0\doteq 0,\beta_0\doteq 0$. 

Therefore, it sufficies to show that if \eqref{1-10-1} holds for some $j\in\mathbb{N}$ then it also holds for $j+1$. Plugging \eqref{1-10-1} into \eqref{if1}, for $t\in[L_jT_0,T)$, we have
\begin{align}
U_1(t)&\geq C\mathrm{e}^{-2t}\int_{L_jT_0}^t\mathrm{e}^{2s}(1+s)^{-\frac{n-1}{2}(p-1)-\frac{\mu}{2}p}(V_1(s))^p\mathrm{d}s \notag\\
&\geq CC_j^p \varepsilon^{p (pq)^j} (1+t)^{-\frac{n-1}{2}(p-1)-\frac{\mu}{2}p-\alpha_jp}\mathrm{e}^{-2t}\int_{L_jT_0}^t\mathrm{e}^{2s}
(s-L_jT_0)^{\beta_jp}\mathrm{d}s. \label{1-10-2}
\end{align}
For $t\in [L_{j+1}T_0,T)$, we have $\frac{t}{\ell_{j+1}}\geq L_jT_0$. Hence, for $t\in [L_{j+1}T_0,T)$, after shrinking the domain of integration from $[L_jT_0,t]$ to $[\frac{t}{\ell_{j+1}},t]$ in \eqref{1-10-2}, we get
\begin{align}
U_1(t) &\geq CC_j^p \varepsilon^{p (pq)^j} (1+t)^{-\frac{n-1}{2}(p-1)-\frac{\mu}{2}p-\alpha_jp}\mathrm{e}^{-2t}\int_{\tfrac{t}{\ell_{j+1}}}^t\mathrm{e}^{2s}
(s-L_jT_0)^{\beta_jp}\mathrm{d}s \notag \\
&\geq CC_j^p \varepsilon^{p (pq)^j}(1+t)^{-\frac{n-1}{2}(p-1)-\frac{\mu}{2}p-\alpha_jp}\left(\frac{t}{\ell_{j+1}}-L_jT_0\right)^{\beta_jp}
\mathrm{e}^{-2t}
\int_{\tfrac{t}{\ell_{j+1}}}^t\mathrm{e}^{2s}\mathrm{d}s \notag \\
&=\frac{CC_j^p}{2\ell_{j+1}^{\beta_jp}} \, \varepsilon^{p (pq)^j} (1+t)^{-\frac{n-1}{2}(p-1)-\frac{\mu}{2}p-\alpha_jp}
(t-L_{j+1}T_0)^{\beta_jp}\left(1-\mathrm{e}^{-2t\big(1-\frac{1}{\ell_{j+1}}\big)}\right). \label{1-10-3}
\end{align}
Now, for $t\geq L_{j+1}T_0$ we estimate
\begin{align*}
1-\mathrm{e}^{-2t\big(1-\frac{1}{\ell_{j+1}}\big)} &\geq 1-\mathrm{e}^{-2L_{j+1}T_0\big(1-\frac{1}{\ell_{j+1}}\big)} =1-\mathrm{e}^{-2L_jT_0(\ell_{j+1}-1)} \\ & \geq 1-\mathrm{e}^{-(\ell_{j+1}-1)} \\ & \geq 1-\left(1-(\ell_{j+1}-1)+\tfrac{1}{2}(\ell_{j+1}-1)^2\right)=(pq)^{-2(j+1)}\left((pq)^{j+1}-\tfrac{1}{2}\right)\\
& \geq(pq)^{-2(j+1)}\left(pq-\tfrac{1}{2}\right),
\end{align*} where we used the inequality $\mathrm{e}^{-\sigma}\leq 1-\sigma+\frac{\sigma^2}{2}$ for $\sigma\geq 0$.

Then, by \eqref{1-10-3}, for $t\in[L_{j+1}T_0,T)$ we have
\begin{align}
U_1(t)\geq \tfrac{C}{2} \left(pq-\tfrac{1}{2}\right) C_j^p \, \ell_{j+1}^{-\beta_jp}(pq)^{-2(j+1)} \varepsilon^{p (pq)^j} (1+t)^{-\frac{n-1}{2}(p-1)-\frac{\mu}{2}p-\alpha_jp}(t-L_{j+1}T_0)^{\beta_jp}. \label{1-10-3,5}
\end{align}
Combining \eqref{1-10-3,5} and \eqref{if2}, for $t\in [L_{j+1}T_0,T)$ we get
\begin{align}
V_1(t) & \geq K\int_{L_{j+1}T_0}^t(1+s)^{-\frac{n-1}{2}(q-1)+\frac{\mu}{2}}(U_1(s))^q\mathrm{d}s \notag \\
&\geq\frac{KC^q (pq-\tfrac{1}{2})^q C_j^{pq} \varepsilon^{(pq)^{j+1}}}{2^q\ell_{j+1}^{\beta_j pq}(pq)^{2q(j+1)}}
\int_{L_{j+1}T_0}^t(1+s)^{-\frac{n+\mu-1}{2}(pq-1)-\alpha_jpq}(s-L_{j+1}T_0)^{\beta_jpq}\mathrm{d}s \notag \\
&\geq \frac{KC^q (pq-\tfrac{1}{2})^q C_j^{pq} \varepsilon^{(pq)^{j+1}}}{2^q\ell_{j+1}^{\beta_j pq}(pq)^{2q(j+1)}}
(1+t)^{-\frac{n+\mu-1}{2}(pq-1)-\alpha_jpq}\int_{L_{j+1}T_0}^t(s-L_{j+1}T_0)^{\beta_jpq}\mathrm{d}s \notag \\
&= \frac{KC^q (pq-\tfrac{1}{2})^q C_j^{pq} \varepsilon^{(pq)^{j+1}}}{2^q\ell_{j+1}^{\beta_j pq}(pq)^{2q(j+1)}(\beta_jpq+1)}
(1+t)^{-\frac{n+\mu-1}{2}(pq-1)-\alpha_jpq}(t-L_{j+1}T_0)^{\beta_jpq+1}. \label{est V1 step j+1}
\end{align}
Hence, we proced \eqref{1-10-1} for $j+1$, provided that
\begin{align}
C_{j+1}&\doteq \frac{KC^q (pq-\frac{1}{2})^q C_j^{pq}}{2^q \ell_{j+1}^{\beta_jpq}(pq)^{2q(j+1)}(\beta_jpq+1)}, \label{Cj rec rel}\\
\alpha_{j+1}&\doteq \frac{n+\mu-1}{2}(pq-1)+pq\alpha_j, \notag\\
\beta_{j+1}&\doteq 1+pq \beta_j.\notag
\end{align}
By using the above relations recursively, we derive explicit representations for $\alpha_j$ and $\beta_j$. Indeed,
\begin{align}
\alpha_j&=\frac{n+\mu-1}{2}(pq-1)+pq \alpha_{j-1} \notag\\
&= \ldots = \frac{n+\mu-1}{2}(pq-1)\sum\limits_{k=0}
^{j-1}(pq)^k+(pq)^j\alpha_0 \notag \\
&=\frac{n+\mu-1}{2}\big((pq)^j-1\big) \label{alphaj}
\end{align}
and, similarly,
\begin{equation}\label{betaj}
\begin{aligned}
\beta_j&=pq\beta_{j-1}+1=\sum\limits_{k=0}
^{j-1}(pq)^k+(pq)^j\beta_0=\frac{(pq)^j-1}{pq-1}.
\end{aligned}
\end{equation}
The next aim is to obtain a lower bound estimate for $C_j$. Note that
\begin{equation*}
\beta_j\leq \frac{(pq)^{j}}{pq-1}
\end{equation*} for any $j\in\mathbb{N}$. Moreover,
\begin{align*}
\lim_{j\to +\infty}\ell_j^{\beta_{j-1}pq}=\lim_{j\to +\infty} \mathrm{e}^{\beta_jpq\ln \ell_j}=\lim_{j\to+\infty}\exp\left(\frac{pq}{pq-1} (1-(pq)^{-j})\frac{\ln \left(1+(pq)^{-j}\right)}{(pq)^{-j}}\right)=\mathrm{e}^{\frac{pq}{pq-1}}>0,
\end{align*}
so, there exists $M=M(p,q)>0$ such that $\ell_j^{\beta_{j-1}pq}< M$ for any $j\in\mathbb{N}$. Consequently, from \eqref{Cj rec rel} we have
\begin{align*}
C_{j}&=\frac{KC^q(pq-\tfrac{1}{2})^q C_{j-1}^{pq}}{2^q\ell_{j}^{\beta_{j-1}pq}(pq)^{2qj} \beta_j}\geq D Q^{-j}C_{j-1}^{pq},
\end{align*}
where $D\doteq KC^q(pq-\frac{1}{2})^q(pq-1)2^{-q} M^{-1}$ and $Q\doteq (pq)^{2q+1}$. Therefore, 
\begin{align*}
\ln C_j&\geq pq\ln C_{j-1}- j\ln Q+\ln D\\
&\geq (pq)^2\ln C_{j-2}-\ln Q \left(j+(j-1)pq\right)+ (1+pq)\ln D\\
&\geq \ldots \geq (pq)^j\ln C_0-\ln Q\sum\limits_{k=0}^{j-1}(j-k)(pq)^k+ \ln D \sum\limits_{k=0}^{j-1}(pq)^k\\
&=(pq)^j\ln C_0-\frac{\ln Q}{pq-1}\left(\frac{(pq)^{j+1}-pq}{pq-1}-j\right)+\frac{(pq)^j-1}{pq-1}\ln D\\
&=(pq)^j\left(\ln C_0-\frac{pq \ln Q}{(pq-1)^2}+\frac{\ln D}{pq-1}\right) +\frac{\ln Q}{pq-1}j-\frac{\ln D}{pq-1}+\frac{pq\ln Q}{(pq-1)^2},
\end{align*} where we used the following identity
\begin{align}\label{identity ln terms}
\sum_{k=0}^{j-1}(j-k)(pq)^k= \frac{1}{pq-1}\left(\frac{(pq)^{j+1}-pq}{pq-1}-j\right).
\end{align}
We fix the smallests $j_0\in\mathbb{N}$ such that $\frac{\ln Q}{pq-1}j_0-\frac{\ln D}{pq-1}+\frac{(pq\ln Q}{(pq-1)^2}>0$. Thus, for any $j\in\mathbb{N}, j\geq j_0$ we deduce that
\begin{equation}\label{logcj}
\begin{aligned}
\ln C_j&\geq(pq)^j\left(\ln C_0-\frac{pq}{(pq-1)^2}\ln Q+\frac{\ln D}{pq-1}\right)= (pq)^j\ln E,
\end{aligned}
\end{equation}
where $E\doteq C_0Q^{-\frac{pq}{(pq-1)^2}}D^{\frac{1}{pq-1}}$.

Combining \eqref{1-10-1},  \eqref{alphaj}, \eqref{betaj} and \eqref{logcj}, for $j\geq j_0$ and $t\in [LT_0,T)$ we obtain
\begin{align}
V_1(t) &\geq C_j \varepsilon^{(pq)^j} (1+t)^{-\alpha_j}(t-LT_0)^{\beta_j} \notag\\
&\geq \exp\left[(pq)^j\left(\ln (E\varepsilon)-\frac{n+\mu-1}{2}\ln(1+t)+\frac{1}{pq-1}\ln(t-LT_0)\right)\right](1+t)^{\frac{n+\mu-1}{2}}
(t-LT_0)^{-\frac{1}{pq-1}}. \label{lowerboundv1}
\end{align}
For $t\geq 2LT_0$ we have $\ln (1+t)\leq\ln(2t)$ and $\ln(t-LT_0)\geq\ln\left( \frac{t}{2}\right)$ (recall that $\ell_0=(2T_0)^{-1}$, so, $2LT_0>1$). Therefore, for $j\in\mathbb{N},j\geq j_0$ and $t\in [2LT_0,T)$, by \eqref{lowerboundv1} we get
\begin{align}\label{lowerboundv1 final}
V_1(t)\geq \exp\left[ (pq)^j\ln \left(\tilde{E}\varepsilon t^{\Theta(n+\mu,p,q)}\right)\right] (1+t)^{\frac{n+\mu-1}{2}}(t-LT_0)^{-\frac{1}{pq-1}},
\end{align}
where $\tilde{E}\doteq 2^{-\frac{n+\mu-1}{2}-\frac{1}{pq-1}}E$. 

We choose $\varepsilon_0>0$ such that $\varepsilon_0\leq  \tilde{E}^{-1} (2LT_0)^{-\Theta(n+\mu,p,q)}$. Since we are considering the subcritical case $\Theta(n+\mu,p,q)>0$, for any $\varepsilon\in(0,\varepsilon_0]$ and any $t>(\tilde{E}\varepsilon)^{-(\Theta(n+\mu,p,q)^{-1}}$, we have $t\geq 2LT_0$ and $\ln (\tilde{E}\varepsilon t^{\Theta(n+\mu,p,q)})>0$, so letting $j\to +\infty$ in \eqref{lowerboundv1 final}, it follows that $V_1(t)$ may not be finite. Hence, we proved that $V_1$ blows up in finite time and we obtained the lifespan estimate $T(\varepsilon)\lesssim \varepsilon^{-(\Theta(n+\mu,p,q))^{-1}}$.

\subsubsection*{Critical case}
This section is devoted to the threshold case $\Theta(n+\mu,p,q)=0$. We modify the sequence for the slicing procedure, following the ideas from \cite{AgKuTa2000}, in order to deal with logarithmic terms in the lower bound estimates. For any $j\in\mathbb{N}$ we set
\begin{equation}\label{case2seq}
\Lambda_j\doteq 1+\frac{1}{T_0}(2-2^{-j}).
\end{equation}
Clearly, $\Lambda_j\uparrow\Lambda \doteq 1+\frac{2}{T_0}$ as $j\to +\infty$.
In this case, we prove the following sequence of lower bound estimates for $V_1$: for any $j\in\mathbb{N}$
\begin{equation}\label{hpcase2}
V_1(t)\geq D_j \varepsilon^{(pq)^j} \left(\ln \frac{t}{\Lambda_jT_0}\right)^{\gamma_j} \qquad \text{for any } t\in [\Lambda_jT_0,T),
\end{equation}
where $\{D_j\}_{j\in\mathbb{N}},\{\gamma_j\}_{j\in\mathbb{N}}$ are sequences of nonnegative real numbers to be determined inductively.
According to Proposition \ref{propo4},  \eqref{hpcase2} is valid for $j=0$ if we set $D_0\doteq K_1$ and $\gamma_0\doteq 0$. Next, we assume that \eqref{hpcase2} holds for some $j\in\mathbb{N}$ and we prove it for $j+1$.

For $t\in [\Lambda_{j+1}T_0,T)$, combining \eqref{hpcase2} and \eqref{if1}, it results
\begin{align*}
U_1(t)&\geq C\int_{T_0}^t\mathrm{e}^{2(s-t)}(1+s)^{-\frac{n-1}{2}(p-1)-\frac{\mu}{2}p} (V_1(s))^p\mathrm{d}s\\
&\geq CD_j^p \varepsilon^{p(pq)^j}\int_{\Lambda_j T_0}^t\mathrm{e}^{2(s-t)}(1+s)^{-\frac{n-1}{2}(p-1)-\frac{\mu}{2}p}\left(\ln \frac{s}{\Lambda_jT_0}\right)^{\gamma_jp}\mathrm{d}s\\
&\geq CD_j^p \varepsilon^{p(pq)^j}(1+t)^{-\frac{n-1}{2}(p-1)-\frac{\mu}{2}p}\mathrm{e}^{-2t}\int_{\tfrac{\Lambda_j t}{\Lambda_{j+1}}}^t\mathrm{e}^{2s}
\left(\ln \frac{s}{\Lambda_jT_0}\right)^{\gamma_jp}\mathrm{d}s\\
&\geq CD_j^p \varepsilon^{p(pq)^j} (1+t)^{-\frac{n-1}{2}(p-1)-\frac{\mu}{2}p}\left(\ln\frac{t}{\Lambda_{j+1}T_0}\right)^{\gamma_jp}\mathrm{e}^{-2t} \int_{\tfrac{\Lambda_jt}{\Lambda_{j+1}}}^t
\mathrm{e}^{2s}\mathrm{d}s\\
&\geq\frac{C}{2} D_j^p \varepsilon^{p(pq)^j} (1+t)^{-\frac{n-1}{2}(p-1)-\frac{\mu}{2}p}
\left(\ln\frac{t}{\Lambda_{j+1}T_0}\right)^{\gamma_jp}\Big(1-\mathrm{e}^{-2(\Lambda_{j+1}-\Lambda_j)
\frac{t}{\Lambda_{j+1}}}\Big).
\end{align*}
Then, for $t\geq \Lambda_{j+1}T_0$ we estimate
\begin{align*}
1-\mathrm{e}^{-2(\Lambda_{j+1}-\Lambda_j)\frac{t}{\Lambda_{j+1}}} & \geq 1-\mathrm{e}^{-2(\Lambda_{j+1}-\Lambda_j)T_0}\\
&\geq
1-\left(1-2(\Lambda_{j+1}-\Lambda_j)T_0+\frac{1}{2}\left(2(\Lambda_{j+1}-\Lambda_j)T_0\right)^2\right)\\
&=2(\Lambda_{j+1}-\Lambda_j)T_0\left(1-(\Lambda_{j+1}-\Lambda_j)T_0\right)\\
&=2^{-(2j+1)}(2^{j+1}-1) \geq2^{-(2j+1)}.
\end{align*}
Therefore, for $t\in [\Lambda_{j+1}T_0,T)$ we obtain
\begin{align}\label{uppcase2u1}
U_1(t)\geq CD_j^p 2^{-2(j+1)} \varepsilon^{p(pq)^j} (1+t)^{-\frac{n-1}{2}(p-1)-\frac{\mu}{2}p}\left(
\ln\frac{t}{\Lambda_{j+1}T_0}\right)^{\gamma_jp}.
\end{align}
Note that $\Theta(n+\mu,p,q)=0$ implies that 
\begin{align*}
-\frac{n-1}{2}(p-1)q-\frac{\mu}{2}pq-\frac{n-1}{2}(q-1)+\frac{\mu}{2}=-\frac{n+\mu-1}{2}(pq-1)=-1.
\end{align*}
 Consequently, plugging \eqref{uppcase2u1} into \eqref{if2}, we have
\begin{align*}
&V_1(t)\geq K\int_{\Lambda_{j+1}T_0}^ts^{-\frac{n-1}{2}(q-1)+\frac{\mu}{2}}(U(s))^q\mathrm{d}s\\
&\geq KC^qD_j^{pq} 2^{-2q(j+1)} \varepsilon^{(pq)^{j+1}}\int_{\Lambda_{j+1}T_0}^t
s^{-1}\left(\ln\frac{s}{\Lambda_{j+1}T_0}\right)^{\gamma_jpq}\mathrm{d}s\\
&=\ KC^qD_j^{pq} 2^{-2q(j+1)}(\gamma_jpq+1)^{-1} \varepsilon^{(pq)^{j+1}} \left(\ln\frac{t}{\Lambda_{j+1}T_0}\right)^{\gamma_jpq+1}.
\end{align*}
Setting
\begin{align*}
D_{j+1} &\doteq \frac{KC^qD_j^{pq}}{2^{2q(j+1)}(1+ pq \gamma_j)},\\
\gamma_{j+1} &\doteq 1+ pq \gamma_j,
\end{align*} we found that \eqref{hpcase2} is valid also for $j+1$.
By using recursively the last relation, we have
\begin{align*}
&\gamma_j=1+pq\gamma_{j-1}=\sum\limits_{k=0}^{j-1}(pq)^k+(pq)^j\gamma_0=\frac{(pq)^j-1}{pq-1},
\end{align*} so, for any $j\in\mathbb{N}$
\begin{align*}
\gamma_j\leq \frac{(pq)^j}{pq-1}.
\end{align*}
Thus,
\begin{align*}
D_j&=\frac{KC^qD_{j-1}^{pq}}{2^{2qj}\gamma_j}\geq \tilde{D} \tilde{Q}^{-j}D_{j-1}^{pq},
\end{align*}
where $\tilde{D}\doteq KC^q(pq-1) $ and $\tilde{Q}\doteq 2^{2q}pq $. Hence,
\begin{align*}
\ln D_j & \geq pq\ln D_{j-1}-j\ln \tilde{Q}+\ln \tilde{D}\\
&\geq \ldots \geq  (pq)^j\ln D_0-\ln \tilde{Q} \sum_{k=0}^{j-1}(j-k)(pq)^k+\ln\tilde{D}\sum_{k=0}^{j-1}(pq)^k \\
&=(pq)^j\ln D_0-\frac{\ln\tilde{Q}}{pq-1}\left(\frac{(pq)^{j+1}-pq}{pq-1}-j\right)+\frac{(pq)^j-1}{pq-1}\ln \tilde{D}\\
&=(pq)^j\left(\ln D_0-\frac{pq\ln \tilde{Q}}{(pq-1)^2}+\frac{\ln\tilde{D}}{pq-1}\right) +\frac{\ln\tilde{Q}}{pq-1}j
+\frac{pq\ln\tilde{Q}}{(pq-1)^2}-\frac{\ln\tilde{D}}{pq-1}.
\end{align*}
Let us denote by $j_2\in\mathbb{N}$ the smallest index such that$\frac{\ln\tilde{Q}}{pq-1}j_2
+\frac{pq\ln \tilde{Q}}{(pq-1)^2}-\frac{\ln\tilde{D}}{pq-1}>0$. Then,  for any $j\in\mathbb{N}$, $j\geq j_2$, it follows that
\begin{equation}\label{supdj}
\begin{aligned}
\ln D_j\geq&(pq)^j\Big(\ln D_0-\frac{pq\ln\tilde{Q}}{(pq-1)^2}+\frac{\ln\tilde{D}}{pq-1}\Big)=(pq)^j\ln N,
\end{aligned}
\end{equation}
where $N\doteq D_0\tilde{D}^{\frac{1}{pq-1}}\tilde{Q}^{-\frac{pq}{(pq-1)^2}}$.

Therefore, combining \eqref{hpcase2} and \eqref{supdj}, for $j\geq j_2$ and $t\in [\Lambda T_0,T)$ we have
\begin{align}
V_1(t)&\geq  D_j \varepsilon^{(pq)^j} \left(\ln \frac{t}{\Lambda T_0}\right)^{\frac{(pq)^j-1}{pq-1}} \notag\\
&\geq \exp\left[(pq)^j\ln\left(N\varepsilon\left(\ln \frac{t}{\Lambda T_0}\right)^{\frac{1}{pq-1}}\right)\right]\Big(\ln \frac{t}{\Lambda T_0}\Big)^{-\frac{1}{pq-1}}. \label{keygs}
\end{align}
Let us remark that 
\begin{align*}
\ln\left(N\varepsilon\left(\ln \frac{t}{\Lambda T_0}\right)^{\frac{1}{pq-1}}\right)>0 \qquad \Leftrightarrow \qquad t> \Lambda T_0\exp\left((N \varepsilon)^{-(pq-1)}\right).
\end{align*} We fix $\varepsilon_0>0$ such that
\begin{align*}
\varepsilon_0 \leq \frac{1}{N}\left(\ln (\Lambda T_0)\right)^{-\frac{1}{pq-1}} \qquad \Leftrightarrow  \qquad \Lambda T_0\leq  \exp\left((N \varepsilon_0)^{-(pq-1)}\right).
\end{align*} Hence, for $\varepsilon\in (0,\varepsilon_0]$ and $t>\exp\left(2(N\varepsilon)^{-(pq-1)}\right)$, we have $t\geqslant \Lambda T_0$ and $\ln\left(N\varepsilon\left(\ln \frac{t}{\Lambda T_0}\right)^{\frac{1}{pq-1}}\right)>0$, so letting $j\to +\infty$ in \eqref{keygs} we find that $V_1(t)$ may not be finite. Then, we proved that $V_1$ blows up in finite time and we derived the upper bound estimate $T(\varepsilon)\leq \exp\left(c \varepsilon^{-(pq-1)}\right) $ where $c\doteq 2 N^{-pq+1}$.

\subsection{Proof of Theorem \ref{theorem2}}

In the scattering producing case, we have the same first lower bound estimates for $V_1$, namely, $V_1(t)\gtrsim \varepsilon$ for any $t\in [0,T)$ (cf. Proposition \ref{V1escase2}). Furthermore, the iteration frame in Proposition \ref{iterationcase2} coincides with the one in Proposition \ref{case1pro} for $\mu=0$. Repeating the same computations as in Subsection \ref{proofofthem1}, we obtain the proof of Theorem \ref{theorem2}.

\paragraph*{Acknowledgments}
Y.-Q. Li gratefully acknowledges the valuable opportunity for overseas study provided by the China Scholarship Council (CSC) (Grant No. 202406860050). He also sincerely thanks the Department of Mathematics at the University of Bari, where this paper was completed, for offering a supportive research environment. A. Palmieri is partially supported by the PRIN 2022 project ``Anomalies in partial differential equations and applications'' CUP H53C24000820006. A. Palmieri is member of the \emph{Gruppo Nazionale per L'Analisi Matematica, la Probabilit\`{a} e le loro Applicazioni} (GNAMPA) of the \emph{Instituto Nazionale di Alta Matematica} (INdAM).

\end{document}